\documentclass{article}

\usepackage{amssymb, amsmath, amsthm}
\usepackage{mypmb}
\usepackage[usenames]{color}
\usepackage{keyval}
\usepackage[colorlinks=true, linktocpage=true, hyperindex=true, linkcolor=blue]{hyperref}

\newtheorem{theorem}{Theorem}[section]
\newtheorem{proposition}[theorem]{Proposition}

\newtheorem{corollary}[theorem]{Corollary}
\newtheorem{claim}[theorem]{Claim}
\theoremstyle{definition}
\newtheorem{definition}[theorem]{Definition}

\newenvironment{display}%
{\refstepcounter{theorem}
 \begin{list}{}{%
  \setlength{\leftmargin}{1em}}
  \item[](\thetheorem) \sl}%
{\end{list}}

\newenvironment{display*}%
{\begin{list}{}{%
  \setlength{\leftmargin}{1em}}
  \item[]\sl}%
{\end{list}}

\newcommand{\n}[1]{}
\newcommand{\theory}[1]{\mathsf {#1}}
\newcommand{\sqr}{\operatorname{\Rightarrow}}
\newcommand{\LK}{\mathbf{LK}}
\newcommand{\tS}{\theory{S}}
\newcommand{\tC}{\theory{C}}
\newcommand{\tF}{\theory F}
\newcommand{\tzf}{\theory{ZF}}
\newcommand{\tzfc}{{\theory{ZFC}}}
\newcommand{\tgb}{\theory{GB}}

\newcommand{\bob}{\mbox{\raisebox{.4ex}{$\ulcorner$}}}
\newcommand{\eob}{\mbox{\raisebox{.4ex}{$\urcorner$}}}
\newcommand{\bmob}{\bob}
\newcommand{\emob}{\eob}
\newcommand{\nmodels}{\not\models}
\newcommand{\gbmodels}{\models^*}
\newcommand{\ngbmodels}{\nmodels^*}
\newcommand{\proves}{\operatorname\vdash}
\newcommand{\nproves}{\operatorname{\nvdash}}
\newcommand{\forces}{\operatorname\Vdash}
\newcommand{\nforces}{\operatorname\nVdash}
\newcommand{\Vcheck}{{\mathsf V}}
\newcommand{\Gcheck}{{\boldsymbol{\mathsf G}}}
\newcommand{\Pcheck}{{\mathsf P}}
\newcommand{\splain}{{\mathsf s}}
\newcommand{\cplain}{{\mathsf c}}
\newcommand{\spr}{{\splain'}}
\newcommand{\cpr}{{\cplain'}}
\newcommand{\sprpr}{{\splain''}}
\newcommand{\spl}{{\splain^+}}
\newcommand{\eqdef}{\operatorname{\overset{\mathrm{def}}{\,\,=\,\,}}}
\newcommand{\teqdef}{$\eqdef$}
\newcommand{\iffdef}{\operatorname{\overset{\mathrm{def}}{\iff}}}
\newcommand{\tiffdef}{$\iffdef$}
\newcommand{\lang}[1]{{\mathcal {#1}}}
\newcommand{\ol}{\overline}
\newcommand{\complet}[1]{{\ol{#1}}}
\newcommand{\cneg}{\operatorname\neg}
\newcommand{\cor}{\operatorname\vee}
\newcommand{\cand}{\operatorname\wedge}
\newcommand{\cimplies}{\operatorname\rightarrow}

\newcommand{\ciff}{\operatorname\leftrightarrow}

\newcommand{\opneg}{\operatorname{\mypmb{\neg}}}
\newcommand{\opand} {\operatorname{\mypmb{\wedge}}}
\newcommand{\opor}{\operatorname{\mypmb{\vee}}}
\newcommand{\opimplies}{\operatorname{\mypmb{\rightarrow}}}
\newcommand{\opiff}{\operatorname{\mypmb{\leftrightarrow}}}
\newcommand{\opforall}{\mypmb{\forall}}
\newcommand{\opexists}{\mypmb{\exists}}
\newcommand{\opin} {\operatorname{\mypmb{\in}}}
\newcommand{\opeq}{\mypmb=}
\newcommand{\opneq}{\mypmb{\neq}}

\newcommand{\opbigand}{\operatorname{\mypmb{\bigwedge}}}
\newcommand{\opbigor}{\operatorname{\mypmb{\bigvee}}}

\newcommand{\axiom}[1]{{\sf#1}}
\newcommand{\Free}{\operatorname{Free}}
\newcommand{\Con}{\operatorname{Con}}
\newcommand{\binop}[1]{\,#1\,}
\newcommand{\vbl}{\textup v}
\newcommand{\Ord}{\operatorname{Ord}}
\newcommand{\alg}[1]{\mathfrak #1}
\newcommand{\lbv}{{[\![}}
\newcommand{\rbv}{{]\!]}}
\newcommand{\bv}[1]{\lbv#1\rbv}
\newcommand{\f}{\hspace{.1em}}

\newcommand{\Inn}{\operatorname{Inn}}

\newenvironment{finfunone}%
{\scriptsize\setlength\arraycolsep{1pt}%
\begin{array}{@{}c@{}}}%
{\end{array}}

\newenvironment{finfuntwo}%
{\scriptsize\setlength\arraycolsep{1pt}%
\begin{array}{@{}cc@{}}}%
{\end{array}}

\newenvironment{finfunthree}%
{\scriptsize\setlength\arraycolsep{1pt}%
\begin{array}{@{}ccc@{}}}%
{\end{array}}

\newenvironment{finfunfour}%
{\scriptsize\setlength\arraycolsep{1pt}%
\begin{array}{@{}cccc@{}}}%
{\end{array}}

\newenvironment{finfunfive}%
{\scriptsize\setlength\arraycolsep{1pt}%
\begin{array}{@{}ccccc@{}}}%
{\end{array}}

\newenvironment{finfunsix}%
{\scriptsize\setlength\arraycolsep{1pt}%
\begin{array}{@{}cccccc@{}}}%
{\end{array}}

\newenvironment{finfunseven}%
{\scriptsize\setlength\arraycolsep{1pt}%
\begin{array}{@{}ccccccc@{}}}%
{\end{array}}

\newenvironment{finfuneight}%
{\scriptsize\setlength\arraycolsep{1pt}%
\begin{array}{@{}cccccccc@{}}}%
{\end{array}}

\newenvironment{newfinfunone}%
{\setlength\arraycolsep{1pt}%
$\begin{array}{@{}c@{}}}%
{\end{array}$}

\newenvironment{newfinfuntwo}%
{\setlength\arraycolsep{1pt}%
$\begin{array}{@{}cc@{}}}%
{\end{array}$}

\newenvironment{newfinfunthree}%
{\setlength\arraycolsep{1pt}%
$\begin{array}{@{}ccc@{}}}%
{\end{array}$}

\newenvironment{newfinfunfour}%
{\setlength\arraycolsep{1pt}%
$\begin{array}{@{}cccc@{}}}%
{\end{array}$}

\newenvironment{newfinfunfive}%
{\setlength\arraycolsep{1pt}%
$\begin{array}{@{}ccccc@{}}}%
{\end{array}$}

\newenvironment{newfinfunsix}%
{\setlength\arraycolsep{1pt}%
$\begin{array}{@{}cccccc@{}}}%
{\end{array}$}

\newenvironment{newfinfunseven}%
{\setlength\arraycolsep{1pt}%
$\begin{array}{@{}ccccccc@{}}}%
{\end{array}$}

\newenvironment{newfinfuneight}%
{\setlength\arraycolsep{1pt}%
$\begin{array}{@{}cccccccc@{}}}%
{\end{array}$}

\newcommand{\ffthree}[1]{\text{\scriptsize\begin{newfinfunthree}#1\end{newfinfunthree}}}

\newcommand{\bone}{\begin{finfunone}}
\newcommand{\eone}{\end{finfunone}}

\newcommand{\btwo}{\begin{finfuntwo}}
\newcommand{\etwo}{\end{finfuntwo}}

\newcommand{\bthree}{\begin{finfunthree}}
\newcommand{\ethree}{\end{finfunthree}}

\newcommand{\bfour}{\begin{finfunfour}}
\newcommand{\efour}{\end{finfunfour}}

\newcommand{\bfive}{\begin{finfunfive}}
\newcommand{\efive}{\end{finfunfive}}

\newcommand{\bsix}{\begin{finfunsix}}
\newcommand{\esix}{\end{finfunsix}}

\newcommand{\bseven}{\begin{finfunseven}}
\newcommand{\eseven}{\end{finfunseven}}

\newcommand{\beight}{\begin{finfuneight}}
\newcommand{\eeight}{\end{finfuneight}}

\definecolor{gray}{rgb}{.7,.7,.7}
\newcommand{\obsub}[1]{{\color{gray}{({\color{black}#1})}}}
\newcommand{\obass}[1]{{\color{gray}{[{\color{black}#1}]}}}

\newcommand{\add}{{}^\curvearrowleft}
\newcommand{\HF}{\operatorname{HF}}
\newcommand{\str}[1]{\mathfrak{#1}}
\newcommand{\Dom}{\operatorname{Dom}}
\newcommand{\subclose}[1]{\ol{#1}}

\newenvironment{denumerate}{%

        \begin{enumerate}}{\end{enumerate}}%

\newcommand{\la}{\langle}
\newcommand{\ra}{\rangle}

\newcommand{\m}[1]{}

\renewcommand{\title}[1]{\begin{center}\huge #1\end{center}}
\renewcommand{\author}[1]{

\begin{center}\large #1\end{center}}
\newcommand{\address}[1]{

\begin{center}\it\large #1\end{center}}
\newcommand{\ead}[1]{

\noindent{{\it email: }\tt #1}\newline}
\newenvironment{keyword}{

\noindent{\it Keywords:} }{\newline}
\newcommand{\sep}{, }

\renewenvironment{abstract}{

\noindent{\bf Abstract}\newline}{\newline}

\begin{document}
\title{Satisfaction relations for proper classes: Applications in logic and set theory}
\author{Robert A. Van Wesep}
\address{1402 Bolton Street\\Baltimore, MD, 21217, USA}
\begin{abstract}
We develop the theory of partial satisfaction relations for structures that may be proper classes and define a satisfaction predicate ($\gbmodels$) appropriate to such structures. We indicate the utility of this theory as a framework for the development of the metatheory of first-order predicate logic and set theory, and we use it to prove that for any recursively enumerable extension $\Theta$ of $\tzf$ there is a finitely axiomatizable extension $\Theta'$ of $\tgb$ that is a conservative extension of $\Theta$. We also prove a conservative extension result that justifies the use of $\gbmodels$ to characterize ground models for forcing constructions.
\end{abstract}
\begin{keyword}
satisfaction\sep proper class\sep class theory\sep conservative extension
\end{keyword}
\ead{rvanwesep@verizon.net}

\section{Introduction}

In discussions of the theory of sets it is a common practice to use the terminology of satisfaction and models informally with reference to structures that are proper classes. An ``inner model'', for example, is informally defined as a proper transitive class $M$ such that $(M;\in)$ ``satisfies'' $\tzf$ (Zermelo-Fraenkel set theory). In the context of a pure set theory, such as $\tzf$, any such reference is necessarily informal for two reasons:
\begin{enumerate}
\item Proper classes do not exist, and can only be referred to as predicates applicable to sets.
\item $\tzf$ is not finitely axiomatizable, so one cannot say that $M$ is an inner model by means of a single sentence relativized to $M$. 
\end{enumerate}
For example, in the context of $\tzf$ the statement
\m{aq}
\begin{display}
\label{aq}
$L$ is a model of $\tzfc$
\end{display}
is understood to stand for the set of all sentences $\theta^L$, where $\theta$ is an axiom of $\tzfc$, and $\theta^L$ is the sentence $\theta$ with all quantified variables are restricted to constructible sets. In the context of a class theory such as the von Neumann-Bernays-G\"odel, or G\"odel-Bernays, theory $\tgb$, one can demonstrate the existence of the class $L$ of constructible sets, and one might expect to be able to formulate (\ref{aq}) as a single sentence
\m{ar}
\begin{display}
\label{ar}
$L \models \tzfc$,
\end{display}
but the interpretation of such a sentence is problematic, inasmuch as it is not possible to prove---in a conservative\footnote{By `conservative' we mean that the class comprehension axiom schema employs only formulas without bound class variables.} class theory such as $\tgb$---that there exists a satisfaction relation for $L$ (or for proper class structures in general).

In this article,  we examine issues related to the existence and use of partial satisfaction relations for proper class structures in the context of conservative class theories with and without the \axiom{Infinity} axiom. In particular, we propose a definition of satisfaction for proper class structures that is weak enough that (\ref{ar}), for example, is a theorem of $\tgb$, but strong enough that (\ref{ar}) implies $\theta^L$ for every axiom $\theta$ of $\tzf$. For the sake of emphasis, in this article we denote this \emph{universal satisfaction predicate} by `$\gbmodels$', but there is no reason its definition (\ref{t}) could not be taken as the primary definition of `$\models$'. By means of $\gbmodels$, the conventional informal use of proper class models may be rendered formally correct with minimal modifications to standard practice. Such notions as elementary substructure and elementary embedding of proper class structures are formalizable in $\tgb$ in a similar way, as are the notions of forcing relation and boolean valuation, all of which are subject to the same limitations as satisfaction in conventional treatments.

Section~\ref{as} is devoted to a brief explication of the notion of satisfaction for proper class structures and related ideas. These are quite straightforward, and we do not suppose that they are entirely new; however, we have not found any systematic treatment of them in the existing literature. In particular, we have not found a definition of $\gbmodels$ or a discussion of its use in set theory. We therefore state the principal theorems and sketch their (straightforward) proofs. We do this primarily in the context of the theories $\tS$ and $\tC$, which we define respectively to be $\tzf$ and $\tgb$ with the \axiom{Infinity} axiom omitted. Note that these theories do not contain $\opneg\axiom{Infinity}$, so they may be extended to $\tzf$ and $\tgb$, respectively, by the addition of \axiom{Infinity}. $\tS + \opneg\axiom{Infinity}$ is---for all practical purposes---Peano arithmetic. The use of satisfaction relations for proper classes in $\tC$ permits an essentially finitary, but nonetheless efficient, development of the theory of first-order predicate logic.\footnote{We do not present this development here beyond what is necessary for the purposes of this article; but the following example is worthy of mention. Consider G\"odel's first incompleteness theorem, which presents a sentence $\sigma$ that says, in effect, that $\sigma$ is not provable in $\tS$ (which we use here for convenience instead of the more customary Peano arithmetic), and states that $\sigma$ is not a theorem of $\tS$, even though it is (by virtue of that fact) true. This is ordinarily proved in a metatheory other than $\tS$, which is often not explicitly characterized, but may generally be considered to be $\tzf$. Then it is stated that the theorem and proof could be given in $\tS$ with the additional hypothesis that $\tS$ is consistent. Therefore, assuming $\Con \tS$, $\tS \nproves \Con\tS$, which is G\"odel's second incompleteness theorem. An actual proof of the first theorem in $\tS$ is generally considered such an ungainly thing that it is not usually given in any detail, despite the fact that its existence is critical to the proof of the second theorem. This defect can be remedied by proving the first theorem in $\tC$, rather than $\tzf$, which is not any harder, and is arguably more natural. Then we can invoke the well known result that $\tC$ is a conservative extension of $\tS$ \eqref{l} to conclude rigorously that a proof of the first theorem in $\tS$ exists, without giving it in detail.} The book \emph{Foundations of Mathematics: A Generalist's Guide}\cite{RVWwtbi1:2012} employs these ideas throughout and demonstrates the clarity and precision they bring to the explication of the foundations of mathematics, from elementary logic to advanced set theory.  

In Section~\ref{at} we prove a generalization of the well known theorem that $\tgb$ is a finitely axiomatizable conservative extension of $\tzf$. This is Theorem~\ref{a}. The universal satisfaction predicate is not involved in the statement of this result, but (a slight variation of) it is intrinsic to the proof.

In Section~\ref{au} we discuss the implications of the use of $\gbmodels$ for the correlation between the forcing relation within a transitive model $M$ of $\tzf$, and the satisfaction relation in a generic extension $M[G]$. If $M$ is a countable set, of course, $p \forces^{\mathbb P} \phi(x_0, \dots, x_{n-1})$ iff for every $M$-generic filter $G$ on a partial order $\mathbb P \in M$, if $p \in G$ then $M[G] \models \phi[x_0^G, \dots, x_{n-1}^G]$. Essentially the same correlation holds in the general case (when $M$ is an uncountable set or a proper class), but it must be stated without reference to generic filters.

This is the method of ``arguing in a generic extension'', which conventionally goes as follows. The theory $\Theta$ (Definition~\ref{aj})---which includes $\tzf$ and also says that the universe is $\Vcheck[G]$, where $\Vcheck$ is an inner model of $\tzf$, and $G$ is a $\Vcheck$-generic filter on $\mathbb P$---holds in $M[G]$ whenever $G$ is an $M$-generic filter on $\mathbb P \in M$, and $\Vcheck$ is interpreted as $M$. ``Arguing in a generic extension'' means proving something in $\Theta$ to demonstrate that something is forced.

If we adopt $\gbmodels$, it is natural to use the theory $\Theta'$ defined in (\ref{ak}), which differs from $\Theta$ in that it includes $\tgb$ and implements the assumption that $\Vcheck$ is an inner model of $\tzf$ as the single sentence \bob$\Vcheck \gbmodels \tzf$\eob, rather than all sentences \bob$\theta^{\Vcheck}$\eob, where $\theta$ is an axiom of $\tzf$. To justify the inference that something is forced from the existence of a proof in $\Theta'$, we show that $\Theta'$ is a conservative extension of $\Theta$. This is Theorem~\ref{ai}.

In Section~\ref{bo}, for the sake of completeness, we define a universal forcing relation and universal boolean valuation function that are analogous to the universal satisfaction relation, by means of which the conventional informal use of forcing relations and boolean valuations for proper classes may be rendered formally correct---again, with minimal modifications to standard practice.

\section{The role of satisfaction for proper classes in set theory and its metatheory}
\m{as}\label{as}

\subsection{Some useful conventions}

Recall the definition of $\tS$ and $\tC$ as $\tzf$ and $\tgb$ with \axiom{Infinity} omitted. We regard $\tC$ (like $\tgb$) as a theory with two sorts of individuals: sets and classes. All sets are classes. A class is \emph{proper} \tiffdef it is not a set. Let $\cplain$ be the signature (similarity type) of this theory, and let $\splain$ be this signature with the class sort omitted. Thus, $\splain$ is the signature of $\tS$ (and $\tzf$). Let $\lang L^{\splain}$ and $\lang L^{\cplain}$ be the languages appropriate to the above signatures. Given a theory $\Theta$, i.e., a class of sentences, let $\complet \Theta$ \teqdef its deductive closure.

It is well known that
\m{k}
\begin{display}
\label{k}
$\tC$ is finitely axiomatizable,
\end{display}
whereas $\tS$ is not.\footnote{This is better known for $\tgb$ and $\tzf$. The finite axiomatization of $\tgb$ as originally given by Bernays\cite{Bernays:1937} yields $\tC$ if \axiom{Infinity} is omitted. A proof that $\tzf$ is not finitely axiomatizable may be given in $\tS + \Con \tzf$ as an easy application of the reflection theorem schema of $\tzf$, i.e. of the theorem of $\tS$ that every instance of the reflection schema is a theorem of $\tzf$. From this it follows that $\tS$ is not finitely axiomatizable (otherwise by adding \axiom{Infinity} we would have a finite axiomatization of $\tzf$); however, this gives the result as a theorem of $\tS + \Con \tzf$ . In what may be viewed as a clever adaptation of the reflection method, Ryll-Nardzewski gave a proof in $\tzf$ that $\tS$ is not finitely axiomatizable\cite{Ryll:1952}. Ryll-Nardzewski's argument is intrinsically infinitary, as it uses the satisfaction relation for $(V_\omega; \in)$. (In fact, the statement of his theorem is intrinsically infinitary, as it applies to arbitrary true extensions of $\tS$, i.e., theories $\Theta \supseteq \tS$ such that $\models^S \Theta$, where $S$ is the full satisfaction relation for $(V_\omega;\in)$.) By way of illustrating some of the ideas presented in this article, we show in the proof of Theorem~\ref{bw} how Ryll-Nardzewski's argument may be (easily) adapted to provide a proof in $\tS + \Con \tS$ that $\tS$ is not finitely axiomatizable.}
\m{l}
\begin{display}
\label{l}
$\tC$ is a conservative extension of $\tS$ in the sense that $\complet{\tC} \cap \lang L^{\splain} = \complet{\tS}$.
\end{display}
(\ref{l}) has a simple infinitary proof (e.g., a $\tzf$-proof), and a considerably more involved finitary proof, (e.g., an $\tS$-proof). Given (\ref{l}), it is appropriate to regard $\tC$ as a \emph{finitary} theory, like $\tS$.

The discussion that follows, particularly the proofs of Theorems~\ref{a} and \ref{ai}, involves a sufficiently intricate interplay of meta- and object theories that it is helpful to use notation for linguistic expressions that distinguishes use and mention more particularly than is often done; although the reader is forewarned that we do not always maintain the highest standard in this regard---a judicious ambiguity sometimes best serves the cause of clarity. Unless otherwise noted, the following discussion takes place in the context of $\tC$. Thus, infinite sets may not exist.

Given a text string that represents an expression $\epsilon$ in an object language, if we flank it with corner quotes, \bmob\dots\emob, we create a string that represents a metalanguage name $\nu$ for $\epsilon$. We use boldface versions of standard typographic symbols for syntactical operations to denote various expression-building operations in any language. Thus, for example, $\phi \opand \psi$ is the conjunction of formulas $\phi$ and $\psi$ in any language. We may extend this to some common predicate and operation symbols, such as those for membership and identity. Thus, if $u = \bmob{}x\emob$, $v = \bmob{}y\emob$, and $w = \bmob{}z\emob$, then $u \opin v \opand v \opin w \opimplies u \opin w$ is \bmob{}if $x \in y$ and $y \in z$ then $x \in z$\emob.

Substitution of terms for variables in expressions is indicated with round brackets. Suppose $\epsilon$ is an expression, and $v_0, \dots, v_{n-1}$ are in $\Free \epsilon$, the set of free variables of $\epsilon$. Suppose $\tau_0, \dots, \tau_{n-1}$ are terms. $\epsilon\big(\ffthree{ v_0 & \dotsm & v_{n-1} \\ \tau_0 & \dotsm & \tau_{n-1} } \big)$ is the expression that results from the indicated substitutions. When it is not necessary to indicate the variables, `$\epsilon(\tau_0, \dots, \tau_{n-1})$' may be used. It is often convenient to use a similar notation where $\epsilon$ is indicated using the corner-quote convention. For example, suppose $\phi$, $\psi$, and $\theta$ are respectively \bmob$x \in y$\emob, \bmob$y \in z$\emob, and \bmob$x \in z$\emob. Then
\m{p}
\refstepcounter{theorem}\begin{equation}
\label{p}
\bmob\text{if $\obsub\phi$ and $\obsub\psi$ then $\obsub\theta$}\emob
\end{equation}
is
\[
\bmob\text{if $x \in y$ and $y \in z$ then $x \in z$}\emob,
\]
the result of substituting the metalanguage terms
\[
\bmob{}x \in y\emob,\ \bmob{}y \in z\emob,\ \bmob{}x \in z\emob 
\]
in the metalanguage term
\[
\bmob\text{if \underline{\quad} and \underline{\quad} then \underline{\quad}}\emob,
\]
where the underscores indicate variables (ranging over formulas) that we do not need to name, as they are always substituted in the indicated fashion. When we indicate such a substitution ``in line'', as in (\ref{p}), we give the brackets a lighter tone than the surrounding text, so as to render them relatively unobtrusive and to distinguish them from round brackets used as grouping indicators.

A similar convention applies to the use of square brackets to indicate assignments of individuals in a structure to variables in an expression for the purpose of valuation (or satisfaction in the case of formulas). Our first use of this is in (\ref{h}). Note that (\ref{h}.1) is a conventional use of square brackets to create the statement that the $\cplain$-formula $D$ is satisfied at the indicated values, viz., $\theta$ and $y$, for its free variables. (\ref{h}.2) is the statement that the formula indicated by the corner-quoted text, with implicit variables in place of the insertions, is satisfied when those variables are assigned the indicated values, viz., $S$ and $n$.

For convenience we suppose that $\splain$ and $\cplain$ have two binary predicate symbols, one to denote membership and one to denote identity. The extended signatures $\spr$ and $\cpr$ have, in addition, a nulary operation symbol (i.e., a constant) to denote the empty set $0$, and a binary operation symbol to denote the \emph{add} operation: $x \add y = x \cup \{y\}$.

`$\HF$' is a defined predicate in $\tS$ characterizing the hereditarily finite sets. In $\tC$, `$\HF$' may be used this way, and also as a constant denoting the class of hereditarily finite sets. Note that in this setting, $\HF = V_\omega$. For each $x$ such that $\HF x$, let $\hat x$ be a specific $\spr$-term (a composition of \bmob$0$\emob{} and \bmob$\add$\emob) whose value is $x$, chosen by some fixed recursive procedure. Call $\hat x$ the \emph{canonical name} of $x$. 

We will use an informal representation of structures, such that
\[
(D_0, D_1, \dots; X_0, X_1, \dots)
\]
is a structure with domains (sorts) $D_0, D_1, \dots$; and predicates and operations $X_0, X_1, \dots$. We regard $(D_0, D_1, \dots; X_0, X_1, \dots)$ as encoding $D_0, \dots, X_0, \dots$ in a way that is applicable to proper classes, as well as sets. 

We suppose that for any signature $\rho$, the expression-building operations for the corresponding language $\lang L^\rho$ are uniformly defined in terms of $0$ and $\add$ in such a way that the rank of any expression is greater than the rank of any of its subexpressions. If $\rho \in \HF$ then $\lang L^\rho \subseteq \HF$.

\subsection{Partial satisfaction relations}

Suppose $\epsilon$ is an expression. Then $\Free \epsilon$ \teqdef the set of free variables of $\epsilon$. $A$ is an $\str S$-assignment for $\epsilon$ \tiffdef $A$ is a finite function into $|\str S|$ such that $\Free \epsilon \subseteq \Dom A$. The value of an expression $\epsilon$ at an $\str S$-assignment $A$ for $\epsilon$ is an element of $|\str S|$ if $\epsilon$ is a term and is a member of $2$, i.e., $\{0, 1\}$, if $\epsilon$ is a formula, with $1$ corresponding to `true' and $0$ to `false'.

Suppose $\Phi$ is a class of $\rho$-expressions, $\subclose \Phi$ \teqdef the class of subexpressions of $\Phi$. We regard an expression as a subexpression of itself, so $\subclose \Phi \supseteq \Phi$. Suppose $\str S$ is a $\rho$-structure. A $\Phi$-valuation function for $\str S$ is a function $F$ such that
\begin{denumerate}
\item $\Dom F$ consists of all $\la \epsilon, A\ra$ such that $\epsilon \in \subclose \Phi$ and $A$ is an $\str S$-assignment for $\epsilon$;
\item $F\la \epsilon, A\ra$ is in $|\str S|$ if $\epsilon$ is a term and in $2$ if $\epsilon$ is a formula;
\item $F$ satisfies the usual recursive definition of valuation.
\end{denumerate}
$F$ is a \emph{partial} valuation function for $\str S$ \tiffdef $F$ is a $\Phi$-valuation for some class $\Phi$ of $\rho$-expressions.

Suppose $\str S$ is a $\rho$-structure. It is straightforward to show in $\tC$ that partial valuations for $\str S$ agree on their common domain.

The following theorem is essentially trivial, but we take the time to state it and sketch the proof to point out where the corresponding proof for the class $\mathcal F^\rho$ of all $\rho$-formulas fails when $\str S$ is a proper class. Here, as elsewhere, we indicate the theory within which a theorem is stated and proved; in this case it is $\tC$. We do the same for definitions.
\m{m}
\begin{theorem}
\label{m}
{\normalfont[$\tC$]} Suppose $\str S$ is a $\rho$-structure. Then there is a unique $\mathcal T^\rho$-valuation function for $\str S$, where $\mathcal T^\rho$ is the class of $\rho$-terms. 
\end{theorem}
\begin{proof} Suppose $\tau$ is a term and $A$ is an assignment for $\tau$. $F$ is a $\la \tau, A\ra$-valuation function for $\str S$ \tiffdef $F$ satisfies the usual definition of valuation function for pairs $\la \tau', A \ra$, where $\tau'$ is a subexpression of $\tau$. Note that $\Free \tau' \subseteq \Free \tau$, so $A$ is an assignment for $\tau'$. A $\la \tau, A\ra$-valuation function is finite and is therefore a set. We now show that for any term $\tau$ and assignment $A$ for $\tau$, there is a unique $\la \tau, A \ra$-valuation function, by supposing toward a contradiction that the class $C$ of terms $\tau'$ with $\Free \tau' \subseteq \Dom A$ for which there is no $\la \tau', A\ra$-valuation function, is nonempty. The definition of $C$ employs only set-quantification, so $C$ exists. Let $\tau_0 \in C$ be of minimal complexity. One easily derives a contradiction.

Now we can define the value of $\tau$ at $A$ as the value assigned to $\tau$ by the unique $\la \tau, A \ra$-valuation function. Again, the quantification over $\la \tau, A \ra$-valuation functions is set-quantification, so the valuation function exists.
\end{proof}

The same argument works for quantifier-free formulas $\phi$, because in this case, an assignment for $\phi$ is also an assignment for any subexpression of $\phi$; but this is not the case if $\phi$ contains quantification. For example, if $\phi = \opexists v\,\psi$, the definition of the value of $\phi$ at $A$ involves the values of $\psi$ at assignments $A \cup \{ (v, a) \}$ for $\psi$, where $a$ ranges over $|\str S|$. Thus, in defining the value of $\phi$ at $A$, it is not enough to quantify over $\la \phi, A \ra$-valuation functions as in the proof of Theorem~\ref{m}; instead, we must quantify over $\{\psi\}$-valuation functions for $\psi$ a subformula of $\phi$. These are proper classes if $\str S$ is a proper class, so we have no justification in $\tC$ for concluding that a class exists such as $C$ in the proof of Theorem~\ref{m}.

It is conventional to speak of valuation of formulas in terms of \emph{satisfaction}. Given a valuation function $F$, the corresponding satisfaction relation $S$ is given by
\[
\la \phi, A \ra \in S \ciff F\la \phi, A \ra = 1.
\]
We adapt the usual symbol for satisfaction to the representation of partial satisfaction by letting $\models^S \phi[A]$ \teqdef $\la \phi, A\ra \in S$ when $S$ is a $\{\phi\}$-satisfaction relation for a structure $\str S$, and $A$ is an $\str S$-assignment for $\phi$. 
The preceding discussion shows that we may not be able to prove in $\tC$ the existence of a full satisfaction relation for a proper class structure. Indeed, it shows that we may not be able to prove the ostensibly weaker statement that for every formula $\phi$ there is a $\{\phi\}$-satisfaction relation.

We do, however, have the following theorem.
\m{n}
\begin{theorem}
\label{n}
{\normalfont [$\tC$]}
Suppose $\str S$ is a $\rho$-structure, $\psi$ and $\psi'$ are $\rho$-formulas, and $\{\psi\}$- and $\{\psi'\}$-satisfaction relations exist for $\str S$. Suppose $\phi$ is obtained from $\psi$ and/or $\psi'$ by a single formula-building operation: $\opneg, \opor, \opand, \opimplies, \opiff, \opexists v, \opforall v$. Then a $\{\phi\}$-satisfaction relation exists.
\end{theorem}
\begin{proof}
Straightforward.
\end{proof}

\m{v}
\begin{definition}
\label{v}
{\normalfont [$\tC$]}
Suppose $\rho \subseteq \HF$ and $n \in \omega$. Let $\Phi^\rho_n$ \teqdef the set of $\rho$-formulas of rank $< n$, i.e., $\mathcal F^\rho \cap V_n$.
\end{definition}
Note that since the expression-building operations are rank-increasing, $\subclose {\Phi^\rho_n} = \Phi^\rho_n$. If $\rho$ is $\HF$, we may use the canonical naming convention to formulate the following theorem, which may be called a metatheorem, inasmuch as it states that an infinite collection of sentences are theorems of $\tC$. The theorem itself is formulated and proved in $\tS$.
\m{o}
\begin{theorem}
\label{o}
{\normalfont [$\tS$]}
Suppose $\rho$ is a signature and $\rho$ is $\HF$. Suppose $n$ is a finite ordinal. Let $\hat \rho$ and $\hat n$ be the canonical names for $\rho$ and $n$. Then $\tC \proves$ \bmob{}for every $\obsub{\hat \rho}$-structure $\str S$ there exists a $\Phi^{\obsub{\hat \rho}}_{\obsub{\hat n}}$-satisfaction relation for $\str S$\emob.
\end{theorem}
\begin{proof}
Let $\rho$ be fixed. We proceed induction on the complexity of formulas using the fact that $\tC$ proves Theorem~\ref{n}.
\end{proof}

\m{r}
\begin{definition}
\label{r}
{\normalfont [$\tC$]}
Suppose $\str S$ is a $\rho$-structure.
\begin{denumerate}
\item $\str S$ is \emph{weakly satisfactory} \tiffdef for every $\rho$-formula $\phi$ there exists a $\{\phi\}$-satisfaction relation for $\str S$, equivalently, for every finite $\Phi$ there is a $\Phi$-satisfaction relation.
\item $\str S$ is \emph{satisfactory} \tiffdef there exists a full satisfaction relation, i.e., an $\mathcal F^\rho$-satisfaction relation, for $\str S$, where $\mathcal F^\rho$ is the class of all $\rho$-formulas.
\end{denumerate}
\end{definition}

\m{w}
\begin{theorem}
\label{w}
{\normalfont [$\tC$]}
Suppose $\str S$ is a structure. If $\str S$ is a set then $\str S$ is satisfactory, i.e., the full satisfaction relation for $\str S$ exists. 
\end{theorem}
\begin{proof}
Straightforward.
\end{proof}

\subsection{The universal satisfaction predicate}

As noted above, the existence of satisfaction relations for proper class structures is problematic, and the following definition is useful in this context.
\m{t}
\begin{definition}
\label{t}
{\normalfont [$\tC$]}
Suppose $\str S$ is a $\rho$-structure.
\begin{denumerate}
\item Suppose $\phi$ is a $\rho$-formula, and $A$ is an $\str S$-assignment for $\phi$. Then $\str S \gbmodels \phi[A]$ \tiffdef for every $\{\phi\}$-satisfaction relation $S$ for $\str S$, $\models^S \phi[A]$.
\item Suppose $\Theta$ is a $\rho$-theory (a class of $\rho$-sentences). $\str S \gbmodels \Theta$ \tiffdef for every $\theta \in \Theta$, $\str S \gbmodels \theta$.
\end{denumerate}
We call $\gbmodels$ the \emph{universal satisfaction predicate}. Note the use of universal, rather than existential, quantification over partial satisfaction relations in the definition of $\gbmodels$.
\end{definition}
The following theorems are relevant. The first is the completeness theorem formulated in the essentially finitary theory $\tC$. 
\m{z}
\begin{theorem}
\label{z}
{\normalfont [$\tC$]}
Suppose $\Theta$ is a consistent theory in a countable signature. Then there is a satisfactory structure $\str S$ such that $\str S \models \Theta$.
\end{theorem}
\begin{proof}
The Henkin construction of a model for a consistent theory $\Theta$ proceeds by defining a complete consistent extension $\Theta'$ of $\Theta$ with witnesses, which are constants in an expanded signature $\spl$. $\str S$ is defined as the structure whose individuals are the $\spl$-terms and whose predicates and operations are given by $\Theta'$. $\Theta'$ also gives the full satisfaction relation for $\str S$.
\end{proof}

\m{bl}
\begin{theorem}
\label{bl}
{\normalfont [$\tC$]}
Suppose $\str S$ is a weakly satisfactory $\rho$-structure, $\Theta$ is a $\rho$-theory, $\sigma$ is a $\rho$-sentence, $\str S \gbmodels \Theta$, and $\Theta \proves \sigma$. Then $\str S \gbmodels \sigma$.
\end{theorem}
\begin{proof}
Suppose $\pi$ is a proof of $\sigma$ from $\Theta$. Let $\Phi$ be the set of formulas occurring in $\pi$, and let $S$ be a $\Phi$-satisfaction relation for $\str S$. Each premise $\theta$ of $\pi$ is then a member of $\Theta$, so $\models^S \theta$. It is straightforward to show that $\models^S \sigma$, so $\str S \gbmodels \sigma$.
\end{proof}

\m{s}
\begin{corollary}
\label{s}
{\normalfont [$\tC$]}
Suppose $\str S$ is a weakly satisfactory $\rho$-structure, $\Theta$ is a $\rho$-theory, and $\str S \gbmodels \Theta$. Then $\Theta$ is consistent.
\end{corollary}
\begin{proof}
Suppose toward a contradiction that $\Theta \proves \sigma \opand \opneg \sigma$ for some $\rho$-sentence $\sigma$. Then $\str S \gbmodels \sigma \opand \opneg \sigma$. Let $S$ be a $\{ \sigma \opand \opneg \sigma \}$-satisfaction relation for $\str S$. Then $\models^S \sigma$ and $\models^S \opneg \sigma$, but the latter implies that $\nmodels^S \sigma$, a contradiction.
\end{proof}

In this connection we note the following theorem.
\m{bm}
\begin{theorem}
\label{bm}
{\normalfont [$\tC$]}
Suppose $\str S$ is a $\rho$-structure and $\sigma$ is a $\rho$-validity, i.e., $\proves \sigma$. Then $\str S \gbmodels \sigma$.
\end{theorem}
The proof is not quite as trivial as the theorem appears to be. To prove it in $\tC$, as opposed to $\tgb$, we use the existence of a deductive system for logic without identity that has the subformula property. For example, let $\LK$ be the \emph{logischer klassischer Kalk\"ul} of Gentzen, as described in \cite[Ch.~1, \S 2]{Takeuti:1975}, and let $\LK^-$ be the same system with the cut rule omitted. The latter has the subformula property, i.e., all formulas appearing in a proof of a sequent are (instances of) subformulas of formulas appearing in the final sequent. By the \emph{cut-elimination theorem} (Gentzen's \emph{Hauptsatz}), any sequent derivable in $\LK$ is derivable in $\LK^-$.

To place this result in the proper perspective from the standpoint of $\tC$, we digress briefly. Consider a fixed language in a signature without identity. The completeness theorem for the sequent calculus is may be taken to be the following assertion:
\m{bq}
\begin{display}
\label{bq}
If a sequent $J = (\Gamma \sqr \Delta)$ is not derivable then there is a an interpretation that does not satisfy $J$, i.e., a structure $\str S$ and an $\str S$-assignment $A$ for $J$ such that $\str S \nmodels J[A]$, i.e., $\str S \models \opbigand \Gamma[A]$ and $\str S \nmodels \opbigor \Delta[A]$.
\end{display}
Here an \emph{interpretation} is ordinarily understood to be a structure $\str S$ and an $\str S$-assignment of all variables.

Working in $\tC$, however, we must be more specific as to the meaning of `interpretation'. Specifically, we define \emph{subvaluation} as in \cite[\S2.5]{RVWwtbi1:2012}. (Briefly, this weakens the notion of partial valuation so that---for example---we may assign the value \emph{true} to $\phi \opor \psi$ without having assigned a value to both $\phi$ and $\psi$; it is enough to have assigned the value \emph{true} to one of these, leaving the other unassigned; whereas in order to assign the value \emph{false} to $\phi \opor \psi$ we must have assigned \emph{false} to both $\phi$ and $\psi$.) We define a \emph{$J$-interpretation} to be a structure $\str S$, an $\str S$-assignment of the free variables of $J$, and an $\str S$-subvaluation that assigns a value to each formula in $J$ at the given assignment. The standard proof of \eqref{bq} (e.g., the proof of Lemma~8.3 in \cite{Takeuti:1975}) yields the following:
\m{br}
\begin{display}
\label{br}
If $J$ is not $\LK^-$-derivable then there is a $J$-interpretation that does not satisfy $J$.
\end{display}
A straightforward modification of this method yields the following:
\m{bs}
\begin{display}
\label{bs}
If $J$ is not $\LK$-derivable then there is a full interpretation that does not satisfy $J$.
\end{display}
This is equivalent to the completeness theorem (\ref{z}) stated above.

The corresponding soundness theorems for $\LK^-$ and $\LK$ are respectively:
\m{bt}
\begin{display}
\label{bt}
If $J$ is $\LK^-$-derivable then every $J$-interpretation satisfies $J$.
\end{display}
\m{bu}
\begin{display}
\label{bu}
If $J$ is $\LK$-derivable then every full interpretation satisfies $J$.
\end{display}
Given \axiom{Infinity}, the distinction between $J$- and full interpretations is irrelevant, as we may restrict our attention to structures that are sets, so that any $J$-interpretation is uniquely extendible to a full interpretation. This yields the standard model-theoretic proof of the cut-elimination theorem: if $J$ is $\LK$-derivable then every full interpretation satisfies $J$, so every $J$-interpretation satisfies $J$, so $J$ is $\LK^-$-derivable. In $\tC$, of course, this proof of cut-elimination is not available, but there are effective proofs (e.g., the proof of Theorem~5.1 in \cite{Takeuti:1975}) that may be rendered in $\tC$, and this yields the following as a theorem of $\tC$:
\m{bv}
\begin{display}
\label{bv}
A sequent $J$ is $\LK^-$-derivable iff $J$ is $\LK$-derivable iff every full interpretation satisfies $J$ iff every $J$-interpretation satisfies $J$.
\end{display}
The following finitary proof of Theorem~\ref{bm}, depending as it does on \eqref{bv}, is therefore another example of the value of the effective proof of cut-elimination.
\begin{proof}[Proof of Theorem~\ref{bm}] 
Suppose $\rho$ is a signature without identity, $\sigma$ is a $\rho$-sentence, $\sigma$ (i.e., the sequent $0 \sqr \{\sigma\}$) is $\LK$-derivable, and $S$ is a $\{\sigma\}$-satisfaction relation for $\str S$. By \eqref{bv}, $\models^S \sigma$. If $\rho$ is a signature with identity, a short additional argument is necessary.
\end{proof}

We will be particularly concerned with satisfaction relations for $(V;\in)$, where $V$ is the class of all sets. Any mention of $V$ as a structure in this article refers to $(V;\in)$ or to an essentially equivalent structure with additional defined predicates or operations.

\m{c}
\begin{theorem}
\label{c}
{\normalfont [$\tS$]}
Suppose $n$ is a finite ordinal, and $\theta \in V_n$ is an $\splain$-sentence. Then $\tC\proves$ \bmob{}The $\Phi^{\splain}_{\obsub{\hat n}}$-satisfaction relation for $V$ exists. Let $S$ be this relation. Then $ \obsub\theta \ciff \models^S \bmob \obsub{\theta}\emob \ciff \models^S  \obsub{\hat\theta} $.\emob.
\end{theorem}
\begin{proof}
By induction on $n$, using Theorem~\ref{o}.
\end{proof}

\m{x}
\begin{theorem}
\label{x}
{\normalfont [$\tC$]}
$(V;\in) \gbmodels \tS$.
\end{theorem}
\begin{proof}
For any individual axiom $\theta \in \tS$, we can prove \bob$(V;\in) \gbmodels \bob\obsub\theta\eob$\eob{} directly. In fact, since $\tC \proves \theta$, \eqref{c} informs us that $\tC \proves$ \bob$(V;\in) \gbmodels \bob\obsub\theta\eob$\eob. Note that the direct proof of \bob$(V;\in) \gbmodels \bob\obsub\theta\eob$\eob{} uses $\theta$ as a premise. Thus, for example, to prove that $(V;\in) \gbmodels  \axiom{Pair}$ we use the fact that pairs exist in $V$, arguing as follows:

Suppose $S$ is a $\{ \theta \}$-satisfaction relation for $V$, where
\[
\theta = \opforall u, u'\, \opexists v\, \opforall w\, ( w \opin v \opiff w \opeq u \opor w \opeq u').
\]
To show that $\models^S \theta$ we must show that
\[
\forall x,x' \in V\, \exists y \in V\, \forall z \in V\, ( \models^S \bob\obass z \in \obass y\eob \ciff \models^S \bob \obass z = \obass x\eob \cor \models^S \bob \obass z = \obass{x'}\eob).
\]
To this end, suppose $x, x' \in V$. Let $y = \{ x, x'\}$. Then $y$ is as desired.

If $\Theta \subseteq \tS$ is an axiom schema, we cannot rely in this way on the fact that $\tC \proves \Theta$ to show that $(V;\in) \gbmodels \Theta$, as our proof must be finite. As it happens, in these cases, it suffices to invoke a single corresponding axiom of $\tC$.

Suppose, for example, that
\[
\theta = \bmob\forall y\, \forall x\, \exists x'\,\forall z\, ( z \in x' \ciff z \in x \cand \obsub\psi(z,y))\emob
\]
is an instance of the \axiom{Comprehension} schema of $\tS$, where $\psi$ is an $\splain$-formula with two free variables. We must show that for every $\{\theta\}$-satisfaction relation $S$ for $V$, $\models^S \theta$. Suppose, therefore, that $S$ is a $\{\theta\}$-satisfaction relation for $V$. We must show that
\[
\forall y\, \forall x\, \exists x'\,\forall z\,\, \models^S \bob z \in x' \ciff z \in x \cand \obsub\psi(z,y)\eob.
\]
Given $y$ and $x$, let $x' = \{ z \in x \mid \models^S \psi[z,y] \}$. Note that the existence of $x'$ as a class follows from a single instance of the \axiom{Comprehension} schema of $\tC$ (with parameters $x, S, \psi, z, y$). That $x'$ is a set follows from the \axiom{Separation} axiom that states that the intersection of a class with a set is a set.

The \axiom{Collection} schema is handled similarly.
\end{proof}

Note that Theorem~\ref{bm} does not permit us to drop the condition of weak satisfactoriness in Theorem~\ref{bl}. In particular, despite Theorem~\ref{x}, we have the following.
\m{bn}
\begin{theorem}
\label{bn}
{\normalfont [$\tS$]}
If $\tS$ is consistent then $\tC \nproves$ \bmob{}for every theorem $\sigma$ of $\tS$, $(V; \in) \gbmodels \sigma$\emob.
\end{theorem}
\begin{proof}
We will prove the contrapositive. Suppose $\tC \proves$ \bmob{}for every theorem $\sigma$ of $\tS$, $(V; \in) \gbmodels \sigma$\emob. Then the following is a proof of $\Con \tS$ in $\tC$.

\bmob{}Suppose toward a contradiction that $\tS$ is inconsistent. Then $\tS \proves\opexists u\,\, u \opneq u$.  By hypothesis, therefore, $(V; \in) \gbmodels \opexists u\,\, u \opneq u$. Let $S$ be a $(\opexists u\,\, u \opneq u)$-satisfaction relation for $(V; \in)$. Then $\models^S \opexists u\,\, u \opneq u$, which is clearly not the case. Hence,  $\tS$ is consistent.\emob{}

Thus, $\tC \proves$  \bmob$\tS$ is consistent\emob. Since $\tC$ is a conservative extension of $\tS$, $\tS \proves$ \bmob$\tS$ is consistent\emob. By G\"odel's second incompleteness theorem, $\tS$ is therefore inconsistent.
\end{proof}

\m{y}
\begin{corollary}
\label{y}
{\normalfont [$\tS$]}
If $\tS$ is consistent then $\tC \nproves$ \bmob$(V; \in)$ is weakly satisfactory\emob.
\end{corollary}
\begin{proof}
Immediate from \eqref{bn} with \eqref{bl}.
\end{proof}

\m{bw}
\begin{theorem}
\label{bw}
{\normalfont [$\tS$]}
If $\tS$ is consistent then $\tS$ is not finitely axiomatizable.
\end{theorem}
Our proof is closely modeled on that of Ryll-Nardzewski\cite{Ryll:1952}, but we work in $\tC$, rather than in $\tzf$. Since $\tC$ is a conservative extension of $\tS$, there is a proof in $\tS$. Since we cannot prove $\Con \tS$ (assuming $\tS$ is consistent) we must specifically assume it. Note that this implies that the theory $\tF = \tS + \opneg\axiom{Infinity}$ is consistent. (Every satisfactory model of $\tS$ has a substructure, viz., its $V_\omega$, that is a satisfactory model of $\tF$.) We will show that $\tF$ is not finitely axiomatizable, from which the theorem follows.
\begin{proof}
Suppose toward a contradiction that $\theta$ is a theorem of $\tF$ such that $\{\theta\} \proves \tF$. Given an existential $\spr$-formula $\psi = \opexists v\,\, \phi(v, v_0, \dots, v_{m-1})$, let $R_\psi(n,n')$ be the formula
\begin{multline*}
\Ord n \opand \Ord n' \opand \opforall v_0, \dots, v_{m-1} \in V_n\, \big(\opexists v\,\, \phi(v, v_0, \dots, v_{m-1})\\
\opimplies \opexists v \in V_{n'}\,\, \phi(v, v_0, \dots, v_{m-1})\big),
\end{multline*}
and given a universal $\spr$-formula $\psi = \opforall v\,\, \phi(v, v_0, \dots, v_{m-1})$, let $R_\psi(n,n')$ be the formula
\begin{multline*}
\Ord n \opand \Ord n' \opand \opforall v_0, \dots, v_{m-1} \in V_n\, \big(\opforall v \in V_{n'}\,\, \phi(v, v_0, \dots, v_{m-1}\\
\opimplies \opforall v\,\, \phi(v, v_0, \dots, v_{m-1})\big).
\end{multline*}
Clearly, for each such formula $\psi$,
\[
\tF \proves \opforall_{\Ord} n\, \opexists_{\Ord} n' > n\,\, R_\psi(n,n').
\]
Let $R$ be the conjunction of the formulas $R_\psi$ for all existential and universal subformulas of $\theta$. Then
\[
\tF \proves \opforall_{\Ord} n\, \opexists_{\Ord} n' > n\,\, R(n,n').
\]
Let $R'(n,k,f)$ be the $\spr$-formula \bob$n$ and $k$ are ordinals, $k > 1$, and $f$ is a function with domain $k$ such that $f(0) = n$ and for all $0 < l < k$, $f(l)$ is the least ordinal $k> f(l-1)$ such that $R(f(l-1), k)$\eob. Let $R''(n,k) = \opexists!f\,\, R'(n,k,f)$. As we have just seen,
\[
\tF \proves \opforall_{\Ord} n\,\, R''(n,2).
\]
It is also clear that
\[
\tF \proves \opforall_{\Ord} n\, \opforall_{\Ord} k > 1\, (R''(n,k) \opimplies R''(n, k+1)).
\]
Hence, by the appropriate instance of the induction schema (\axiom{Foundation}) of $\tS$,
\[
\tF \proves \opforall_{\Ord} n\, \opforall_{\Ord} k > 1\,\, R''(n,k).
\]

Now let $\sprpr$ be the expansion of the signature $\spr$ by the addition of a single constant symbol $N$, and let $\Theta$ be the $\sprpr$-theory $\tF \cup \{\text{\bob $\Ord N$\eob} \} \cup \{ \bob N > \obsub{\hat n} \eob \mid n \in \omega\}$. Recall that for any $x \in \HF$, $\hat x$ is a canonical $\spr$-term denoting $x$. Since $\tF$ is consistent, so is $\Theta$. By the completeness theorem \ref{z}, there is a satisfactory structure $\str S$ such that $\str S \models \Theta$.

The ordinals of $\str S$ are linearly ordered, with an initial segment isomorphic to $(\omega;<)$. We suppose for simplicity that this segment (with its order) is actually $(\omega; <)$.

Let $\bar N = N^{\str S}$ and let $\bar f$ be the unique element of $|\str S|$ such that $\str S \models R'[\bar N,\bar N, \bar f]$. Let $f$ be the corresponding function from the ordinals in $\str S$ preceding $\bar N$ into $\Ord^{\str S}$.

Let $\str S'$ be the substructure of $\str S$ such that $|\str S'| = \bigcup_{n \in \omega} V_{f(n)}^{\str S}$. By construction, $\str S'$ is an initial segment of $\str S$ containing $\bar N$ and is a $\{\theta\}$-elementary substructure of $\str S$. Hence, $\str S' \models \theta$, so $\str S' \models \tF$, and
\[
\str S' \models \opexists f'\,\, R'(N, N, f').
\]
Let $\bar f'$ be the unique element of $|\str S'|$ such that
\[
\str S' \models R'[\bar N, \bar N, \bar f'],
\]
and let $f'$ be the corresponding function from the ordinals in $\str S'$ preceding $\bar N$ into $\Ord^{\str S'}$. Note that $\Dom f' = \Dom f$. Since $\str S'$ is a $\{\theta\}$-elementary substructure of $\str S$, $f' = f$. This contradicts the fact that $f \restriction \omega$ is cofinal in $\Ord^{\str S'}$.
\end{proof}

\subsection{Inner models, elementary embeddings, etc.}

The theory of partial satisfaction outlined above provides simple and useful definitions of some of the more problematic notions in the metatheory of set theory. For example, an \emph{inner model} may be defined as a transitive proper class $M$ such that $(M; \in) \gbmodels \tzf$.\footnote{Note that according to this definition, it is a theorem of $\tgb$ that $V$ is an inner model of $\tzf$ and $L$ is an inner model of $\tzfc$.} An elementary substructure of a $\rho$-structure $\str S$ (which may be a proper class) may be defined as a substructure $\str S'$ of $\str S$ with the property that for every $\rho$-formula $\phi$ and every $\{\phi\}$-satisfaction relation $S$ for $\str S$, the restriction of $S$ to $\str S'$ is a $\{\phi\}$-satisfaction relation for $\str S'$. A function $j : \str S \to \str T$ is elementary iff it is an isomorphism of $\str S$ with an elementary substructure of $\str T$.
 
Given the simplicity and utility of these ideas, it is somewhat surprising that they have not gained greater currency in the exposition of the metatheoretical aspects of set theory, especially given that fact that it is standard practice to employ proper classes for this purpose---albeit informally as the extensions of formulas (in which quantification is necessarily restricted to sets, since only sets are actually supposed to exist). Since the definition of $\gbmodels$, etc., involves quantification over classes, the explanation may reside in an instinctive aversion to such quantification, deriving from a knowledge of the paradoxes that lurk beyond the pale.

Adherence to the convention that class variables are not to be quantified is illustrated in \cite{Solovay:1978}, which provides a formal definition of \emph{inner model} by means of a particular conjunction $\sigma$ of axioms of $\tzfc$ and a formula $\Inn(M)$ such that
\begin{enumerate}
\item $\Inn(M)$ is \bob$\obsub{\sigma^M}$ and $\obsub M$ is a transitive proper class\eob, and
\item for every theorem $\theta$ of $\tzfc$, $\tzfc \proves \Inn(M) \opimplies \theta^M$.
\end{enumerate}
Another approach is to define an inner model as a transitive class that is almost universal and closed under G\"odel operations\cite[p.182]{Jech:2003}. With any of these approaches other \emph{ad hoc} arrangements are necessary to deal with elementary embeddings of inner models and related notions.

In \cite[Appendix~X]{Levy:1979} Levy makes the exclusion of class quantification explicit in his description of the logical system $P^*$, which extends the language $\lang L^{\splain}$ of pure set theory by the addition of class variables, which cannot be quantified; class terms $\{ x \mid \Phi(x)\}$, where $x$ is a set variable and $\Phi$ is any formula; and axioms that define the relations of membership and equality between sets and class terms. The stipulation that class variables are not to be quantified allows for a relatively easy syntactical proof (compared to \cite{Shoenfield:1954}) that $P^*$ is a conservative extension of the theory $P$, which is essentially the axiom of extensionality in $\lang L^{\splain}$.

\section{Conservative extension of set theories to finitely axiomatizable class theories}
\m{at}\label{at}

\begin{theorem}
\label{a}\m{a}
{\normalfont [$\tC$]} For any recursively enumerable extension $\Theta$ of $\tS$, there is a finite extension $\Theta'$ of $\tC$ such that $\complet{\Theta'}  \cap \lang L^{\splain} = \Theta$.
\end{theorem}
Note that this is the strongest possible theorem along these lines, inasmuch as the deductive closure of a finite theory is necessarily recursively enumerable. Note also that it is proved in $\tC$, which does not have \axiom{Infinity}. Since $\tC$ is a conservative extension of $\tS$, if we stated the straightforward translation of the theorem into $\lang L^{\splain}$ it would be a theorem of $\tS$.

This should not be confused with the superficially similar results of Kleene\cite{Kleene:1952} and Craig and Vaught\cite{Craig:1958}, which produce a finitely axiomatizable conservative extension of a recursively enumerable $\rho$-theory $\Theta$ (which is required to have only infinite models) by introducing additional predicates with axioms asserting that they represent the $\rho$-language and satisfaction predicate, and that all sentences of $\Theta$ are true. 

To highlight the issues surrounding the existence of satisfaction relations and the implications for provability, we first give a proof of the theorem in $\tzf$, i.e., using \axiom{Infinity}, and we then show how to accomplish it in $\tC$.

\begin{proof}[Infinitary proof of Theorem~\ref{a}] Suppose $\Theta$ is a recursively enumerable extension of $\tS$. Let $D$ be an $\splain$-formula with two free variables, all of whose quantifiers are bounded, such that for all $x \in \HF$, $x \in \Theta \ciff \exists y \in \HF\,\, D(x,y)$.

Let $\Theta' = \tC \cup \{ \theta'\}$, where $\theta' =$
\m{ad}
\begin{display}
\label{ad}
\normalfont\bmob{}$\forall S\, \forall n \in \omega\, \forall x \in V_n\,$(if $\exists y \in V_n\,\, \obsub D(x,y)$ and $S$ is the $\Phi^{\splain}_n$-satisfaction relation for $V$, then $\models^S x$)\emob.
\end{display}
$\theta'$ says roughly that for every sentence $x \in \Theta$, $V \gbmodels x$, making use of a slight modification of the universal satisfaction predicate as given by Definition~\ref{t}. Since Theorem~\ref{a} does not mention satisfaction, we are free to define it as we wish within the proof. 

\m{e}
\begin{claim}
\label{e}
$\complet{\Theta'} \cap \lang L^{\splain} \supseteq \complet \Theta$.
\end{claim}
{\sc Proof.} Suppose $\theta \in \Theta$. Let $n \in \omega$ be such that $\theta \in V_n$ and there exists $y \in V_n$ such that $D(\theta, y)$. It is easy to show that $\tC \proves$ \bmob$\obsub{\hat \theta} \in V_{\obsub{\hat n}} \cand \obsub{\hat y} \in V_{\obsub{\hat n}} \cand \obsub{D(\hat \theta, \hat y)}$\emob. Thus, $\Theta' \proves$ \bmob{}for all $S$, if $S$ is the $\Phi^{\splain}_{\obsub{\hat n}}$-satisfaction relation for $V$ then $\models^S \obsub{\hat \theta}$\emob. It follows from Theorem~\ref{c} that $\Theta' \proves \theta$.\qed(\ref{e})

\m{f}
\begin{claim}
\label{f}
$\complet{\Theta'} \cap \lang L^{\splain} \subseteq \complet \Theta$.
\end{claim}
{\sc Proof.} Suppose $\sigma \in \complet{\Theta'} \cap \lang L^{\splain}$. We will use the completeness theorem to show that $\Theta \proves \sigma$ by showing that $\sigma$ holds in any satisfactory countable model of $\Theta$. Suppose, therefore, that $\str M = (M; E)$ is a satisfactory countable model of $\Theta$. Like all models of $\tS$, $\str M$ has an initial segment that is isomorphic to $(\HF; \in)$, and in the interest of efficiency we arrange that this initial segment actually is $(\HF; \in)$. Since we are working in $\tzf$, there are plenty of sets outside $\HF$.

We now extend $\str M$ to a model $\str M' = (M, M'; E')$ of $\tC$ by adding, for each subset $A$ of $M$ definable over $\str M$ (from a parameter in $M$) that is not already the $E$-extension of a member of $M$, a new element whose extension is $A$. (These are the proper classes.) We will identify each $X \in M' \setminus M$ with $\{ x \in M \mid x \binop{E'} X \}$.

This is the standard construction for the infinitary proof that $\tC$ is a conservative extension of $\tS$. Since we are working in $\tzf$, $\str M'$ is a set, and therefore is satisfactory, so all references to satisfaction in $\str M'$ may be understood in the usual way. It is easy to show that $\str M' \models \tC$.

\m{g}
\begin{claim}
\label{g}
$\str M' \models \theta'$.
\end{claim}
{\sc Proof.} Remember that we have arranged that $\HF$ is an initial segment of $\str M$ and therefore of $\HF^{\str M'}$. Suppose toward a contradiction that $n \in \omega^{\str M}$, $\theta, y \in (V_n)^{\str M}$, and $S \in M'$ are such that
\m{h}
\begin{display}
\label{h}
$\str M' \models$
\begin{enumerate}
\item $D[\theta,y]$,
\item \bmob$\obass S$ is the $\Phi^{\splain}_{\obass n}$-satisfaction relation for $V$\emob, and
\item \bmob$\nmodels^{\obass S} \obass \theta$\emob.
\end{enumerate}
\end{display}
Suppose first that $n$ is in the standard part of $\omega^{\str M}$. Then $n \in \omega$; $\theta, y \in V_n$; and $D(\theta,y)$; so $\theta \in \Theta$. Thus, $\str M \models \theta$. Also, $\{ x \in M \mid \str M' \models \bmob\obass x \in \obass S\emob \}$ is the $\Phi^{\splain}_n$-satisfaction relation for $\str M$. Thus, by virtue of (\ref{h}.3), $\str M' \models \opneg \theta$, contrary to the fact that $\str M \models \theta$.

Thus, $n$ is in the nonstandard part of $\omega^{\str M}$. Note that $V_\omega \subseteq V_n^{\str M}$, and it is straightforward to show by induction within $V_\omega$ that $S$ restricted to formulas in $V_\omega$ is the satisfaction relation for $\str M$, i.e., for any $\splain$-formula $\psi$ and $\str M$-assignment $A$ for $\psi$,
\m{i}
\refstepcounter{theorem}
\begin{equation}
\label{i}
\str M' \models \bmob\models^{\obass S} \obass \psi [\obass A] \emob \ciff \str M \models \psi[A].
\end{equation}
By virtue of the construction of $\str M'$, $S$ is definable over $\str M$ from a parameter in $M$. This allows us to apply the G\"odel-Tarski theorem on the undefinability of truth, the proof of which we reprise for the present application.

Let $\phi \in V_\omega$ and $z \in M$ be such that $\phi$ is an $\splain$-formula with free variables $\vbl_0, \vbl_1, \vbl_2$, and for all $\psi, A \in M$ such that $\str M \models$ \bmob$\obass \psi$ is an $\splain$-formula and $\obass A$ is an assignment of its free variables\emob,
\[
\str M' \models \bmob\models^{\obass S} \obass \psi [\obass A] \emob
 \ciff \str M \models \phi[\psi, A, z].
\]
It is easy to obtain from $\phi$ an $\splain$-formula $\phi'$ such that for all $\psi, a \in M$ such that $\str M \models$ \bmob$\obass \psi$ is an $\splain$-formula with free variables $\vbl_0, \vbl_1$\emob,
\m{aa}
\refstepcounter{theorem}\begin{equation}
\label{aa}
\str M' \models \bmob\models^{\obass S} \obass\psi [\obass a, \obass z]\emob
 \ciff \str M \models \phi'[\psi, a, z].
\end{equation}
Let $\psi = \opneg \phi'(\vbl_0, \vbl_0, \vbl_1)$. Note that $\phi', \psi \in V_\omega$, so by virtue of (\ref{aa}) and (\ref{i}),
\[
\begin{split}
\str M \models \psi[\psi, z]
 &\ciff \str M \models \opneg \phi'[\psi, \psi, z]\\
 &\ciff \str M' \models \bmob \nmodels^{\obass S} \obass\psi [ \obass\psi, \obass z ] \emob\\
 &\ciff \cneg \str M \models \psi[\psi, z]\\
 &\ciff \str M \models \opneg \psi[\psi, z].
\end{split}
\]
This contradiction establishes the claim.\qed(\ref{g})

Recall that we have supposed that $\sigma \in \complet{\Theta'} \cap \lang L^{\splain}$, and that $\str M$ is an arbitrary satisfactory countable model of $\Theta$. We have constructed the satisfactory structure $\str M'$ with the same ``sets'' as $\str M$, and we have shown that $\str M' \models \Theta'$. Hence $\str M' \models \sigma$, so $\str M \models \sigma$. Thus, $\Theta \proves \sigma$, i.e., $\sigma \in \complet\Theta$.\qed(\ref{f})

Claims~\ref{e} and \ref{f} together establish the theorem.
\end{proof}

\begin{proof}[Finitary proof of Theorem~\ref{a}] The finitary\footnote{Note that we construe `finitary' broadly to include $\tC$. As pointed out in the remark following the statement of the theorem, the theorem and proof could be given in $\tS$.} proof is identical to the preceding argument through the proof of Claim~\ref{e} but differs beginning with the proof of Claim~\ref{f}. We restate the claim here.
\m{ab}
\begin{claim}
\label{ab}
$\complet{\Theta'} \cap \lang L^{\splain} \subseteq \complet \Theta$.
\end{claim}
{\sc Proof.} Suppose $\sigma \in \complet{\Theta'} \cap \lang L^{\splain}$, and suppose toward a contradiction that $\Theta \nproves \sigma$, i.e., $\Theta \cup\{\opneg\sigma\}$ is consistent. Using Theorem~\ref{z}, let $\str M = (M; E)$ be a satisfactory structure such that $\str M \models \Theta \cup\{\opneg\sigma\}$, and let $T$ be the satisfaction relation for $\str M$.

Like all models of $\tS$, $\str M$ has an initial segment that is isomorphic to $(\HF; \in)$. Since it is possible that $V = \HF$, we do not suppose that $\HF$ is itself an initial segment of $\str M$; on the contrary, we arrange that $V \setminus |\str M|$ is a proper class, and that $M \cap M_1 = 0$, where $M_1$ is the class of $\la \phi, a \ra$ such that $\phi$ is an $\splain$-formula with two free variables $\vbl_0, \vbl_1$; $a \in M$; and $\{ x \mid\, \models^T \phi[x,a] \}$ is not the $E$-extension of a member of $M$.

Let $E_1 = \{ \la x, \la \phi, a \ra \ra \mid \la \phi, a \ra \in M_1 \cand \models^T \phi[x,a] \}$. Let $M' = M \cup M_1$ and $E' = E \cup E_1$, and let $\str M' = (M, M'; E')$. $\str M'$ is the canonical expansion of $\str M$ to a model of $\tC$; however, whereas in $\tzf$ we could let $M'$ consist of definable subsets of $M$, we now let $M'$ consist of definitions of subclasses of $M$. The reason, of course, is that $M$ may be a proper class. As a consequence, distinct elements of $M'$ may have the same $E'$-extension. We nevertheless define
\[
=^{\str M'} = \{ \la x, y \ra \in M' \times M' \mid \forall z \in M\, (\la z, x \ra \in E' \ciff \la z, y \ra \in E' \ra\}. 
\]
The equivalence classes of $=^{\str M'}$ may be proper classes, so the reduction of $\str M'$ to a standard model of logic with identity (i.e., one for which the identity predicate is interpreted as the identity relation) would be slightly more involved than usual. It could be done, but there is no need, so we don't bother.

Let $\Psi$ be the class of $\cplain$-formulas with class-quantifier depth at most 2 (counting quantifier depth in terms of alternations of existential and universal quantification). Note that we impose no restriction on set quantifiers. All axioms of $\tC$ may be formulated as sentences in $\Psi$. Using the fact that $T$ is the full satisfaction relation for $\str M$, we can easily show that the $\Psi$-satisfaction relation exists for $\str M'$ (essentially by direct definition), and we let $T'$ be this relation. Note that $T' \supseteq T$. It is straightforward to show that for every $\theta$ that is an axiom of $\tC$ or an axiom of identity, $\models^{T'} \theta$.

We now state and prove the analog of Claim~\ref{g} for the present situation. Note that $\theta'$, as defined in (\ref{ad}), is in $\Psi$.
\m{ac}
\begin{claim}
\label{ac}
$\models^{T'} \theta'$.
\end{claim}
{\sc Proof.} Suppose toward a contradiction that $\nmodels^{T'} \theta'$. By the definition of $T'$, this implies that there exist $S \in M'$; $n \in \omega^{\str M}$; and $\theta, y \in (V_n)^{\str M}$ such that (\ref{h}) holds, where we may substitute $\models^{T'}$ for $\str M' \models$.

As before, we first suppose $n$ is in the standard part of $\omega^{\str M}$. Let $H$ be the standard part of $\HF^{\str M}$, and let $\iota : H \to \HF$ be the (unique) isomorphism. For notational convenience, let $\bar x = \iota(x)$ for $x \in H$, and extend this notation to assignment functions, so that $\bar A(\bar v) = A(v)$ for any assignment $A$ and variable $v$ in the sense of $\str M$. Then $\bar n \in \omega$; $\bar \theta, \bar y \in V_{\bar n}$; and $D(\bar\theta,\bar y)$; so $\bar\theta \in \Theta$. Thus, $\str M \models \bar\theta$.

It is straightforward to show by induction on complexity that $S$ agrees with $T$ (the full satisfaction relation for $\str M$), i.e., $\models^{T'}$ \bmob$\models^{\obass S} \obass \phi[\obass A] $\emob{} iff $\models^T \bar \phi[\bar A]$, i.e., $\str M \models \bar \phi[\bar A]$.  Thus, by virtue of (\ref{h}.3), $\str M \nmodels \bar\theta$, contrary to the fact that $\str M \models \bar\theta$.

$n$ is therefore in the nonstandard part of $\omega^{\str M}$. As before, it is straightforward to show by induction within $V_\omega$ that $S$ restricted to formulas $\psi$ in the standard part of $\str M$ agrees with $T$, i.e.,
\m{ae}
\refstepcounter{theorem}\begin{equation}
\label{ae}
\models^{T'} \bmob\models^{\obass S} \obass \psi[\obass A]\emob \ciff \str M \models \bar\psi[\bar A].
\end{equation}
As before, since $S$ (actually, $\{ x \in M \mid \la x, S \ra \in E' \}$) is definable over $\str M$, there exist $z \in M$ and $\phi' \in H$ such that $\bar \phi'$ is an $\splain$-formula with free variables $\vbl_0, \vbl_1, \vbl_2$---and for all $\psi, a \in M$ such that $\str M \models$ \bmob$\obass \psi$ is an $\splain$-formula with free variables $\vbl_0, \vbl_1$\emob,
\m{af}
\refstepcounter{theorem}\begin{equation}
\label{af}
\models^{T'} \bmob\models^{\obass S} \obass \psi[ \obass a, \obass z]\emob
 \ciff \str M \models \bar\phi'[\psi, a, z].
\end{equation}
Let $\psi = \opneg \phi'(\vbl_0, \vbl_0, \vbl_1)$ in the sense of $\str M$. Then $\psi \in H$ and $\bar \psi = \opneg \bar \phi'(\vbl_0, \vbl_1, \vbl_2)$, so by virtue of (\ref{af}) and (\ref{ae}),
\[
\begin{split}
\str M \models \bar\psi[\psi, z]
 &\ciff \str M \models \opneg \bar\phi'[\psi, \psi, z]\\
 &\ciff \models^{T'} \bmob\nmodels^{\obass S} \obass \psi[\obass \psi, \obass z]\emob\\
 &\ciff \cneg \str M \models \bar\psi[\psi, z]\\
 &\ciff \str M \models \opneg \bar\psi[\psi, z].
\end{split}
\]
This contradiction establishes the claim.\qed (\ref{ac})

Recall that we have supposed that $\sigma \in \complet{\Theta'} \cap \lang L^{\splain}$, i.e., $\sigma$ is an $\splain$-sentence and $\Theta' \proves \sigma$; and we are attempting to show that $\Theta \proves \sigma$, i.e., $\Theta \cup \{  \opneg \sigma \}$ is inconsistent. We have supposed toward a contradiction that $\str M$ is a satisfactory structure such that $\str M \models \Theta \cup \{ \opneg \sigma\}$. We have constructed $\str M'$ such that $\str M'$ has the same ``sets'' as $\str M$, and $\models^{T'} \Theta' \cup \{\opneg \sigma \}$, with the $\Psi$-satisfaction relation $T'$. To complete the proof we must derive a contradiction from the fact that $\Theta' \proves\sigma$.

Theorem~\ref{s} is not available, because we have not shown that $\str M'$ is weakly satisfactory. It is sufficient, however, that $\str M$ is satisfactory. We make use of the methodology of the finitary proof of (\ref{l}),\footnote{We are paraphrasing somewhat Shoenfield's original finitary proof of this result in \cite{Shoenfield:1954}. A proof that is directly applicable to the present situation may be found in \cite[Chap.~2]{RVWwtbi1:2012}.} which proceeds by showing how to transform a $\tC$-proof $\pi$ of an $\splain$-sentence $\nu$ into an $\tS$-proof $\pi'$ of $\nu$ by the systematic elimination of class variables in favor of class constants, and the elimination of the latter in favor of expressions involving $\splain$-formulas essentially serving as definitions. In this process, for each class constant $C$ we define an appropriate $\splain$-formula $\phi_C$ with one free variable, and we replace each expression $\tau \opin C$ by $\phi_C(\tau)$. The premises of $\pi$ are axioms of $\tC$, each of which is replaced in $\pi'$ by finitely many instances of axioms of $\tS$. $\nu$ is not affected by this transformation.

If we apply this procedure to a $\Theta'$-proof of $\sigma$, we arrive at an $\splain$-proof with premises that are instances of $\tS$-axioms and sentences $\theta^\phi$ obtained from $\theta'$ by omitting the universal quantification of $S$ and replacing each expression $\tau \opin S$ in (\ref{ad}) (which we imagine to be written out in full) by $\phi(\tau)$, where $\phi$ is an $\splain$-formula with one free variable. By construction, for any $\splain$-formula $\phi$ with one free variable, $\{ x \in M \mid \str M \models \phi[x] \}$ is $\{ x \in M \mid \models^{T'} \bmob\obass x \in \obass S\emob \}$ for some $S \in M'$. Since $\models^{T'}$ is the $\Psi$-satisfaction relation for $\str M'$, where $\Psi$ is the class of $\cplain$-formulas with class-quantifier depth at most 2, and $\models^{T'} \theta'$, it follows that $\models^{T'} \theta^\phi$. Hence $\str M \models \theta^\phi$. (We could also use the fact that $\models^{T'} \tC \cup \{ \theta'\}$, and each $\theta^\phi$ has a proof from $\tC \cup \{ \theta' \}$ all of whose formulas are in $\Psi$.)

Thus, we have an $\splain$-proof of $\sigma$ from premises that are true in $\str M$, from which it follows that $\str M \models \sigma$. This contradiction establishes Claim~\ref{ab} and with it the theorem.
\end{proof}

\section{Universal satisfaction and forcing}
\m{au}\label{au}

The basic construction in the theory of forcing is that of a \emph{generic extension} $M^{\mathbb P, G}$ of a transitive model $M$ of $\tzf$ by an $M$-generic filter $G$ on a partial order $\mathbb P \in M$, where $M^{\mathbb P, G}$ is the structure that interprets each \emph{forcing term} $x \in M^{\mathbb P}$ as $x^G$. $M$ is referred to as the \emph{ground model}. The usefulness of this construction derives from the existence of a \emph{forcing relation} $\forces^{M, \mathbb P}$ with the following properties:
\begin{denumerate}
\item For any sentence $\sigma$ of the \emph{forcing language} $\lang L^{M, \mathbb P}$,
\begin{denumerate}
\item if for every $p \in |\mathbb P|$ there is an $M$-generic filter $G$ on $\mathbb P$ with $p \in G$, then for any $p \in |\mathbb P|$, $p \forces^{M, \mathbb P} \sigma$ iff for every $M$-generic filter $G$ on $\mathbb P$, if $p \in G$ then $M^{\mathbb P, G} \models \sigma$; and
\item for any $M$-generic filter $G$ on $\mathbb P$, $M^{\mathbb P, G} \models \sigma$ iff for some $p \in G$, $p \forces^{M, \mathbb P} \sigma$.
\end{denumerate}
\item $\forces^{M, \mathbb P}$ is, in a suitable sense, definable over $M$, i.e., over the structure $(M;\in)$.
\end{denumerate}
Note that as far as $M$ is concerned, $\forces^{M, \mathbb P}$ is $\forces^{V, \mathbb P}$, where $V$ is the class of all sets, so proper classes as ground models are an inescapable feature of the theory of forcing, with $V$ as the paradigm. Given a partial order $\mathbb P$, we let $\forces^{\mathbb P}$ be $\forces^{V, \mathbb P}$.

When the ground model $M$ is a proper class, the assumption that $M$ models $\tzf$ cannot be formulated as \bmob$M \models \tzf$\emob, as the full satisfaction relation for $M$ may not exist. In the context of $\tzf$, we may implement this assumption by positing \bmob$\theta^M$\emob{} for every axiom $\theta$ of $\tzf$ ($M$ being given by means of a defining formula). In the context of $\tgb$ we have the option of formulating this assumption as \bmob$M \gbmodels \tzf$\emob. This use of $\gbmodels$ in the context of forcing raises an issue that does not arise in other applications of $\gbmodels$ in set theory, which we will describe presently, and which we will settle by means of a conservative extension result, Theorem~\ref{ai}, somewhat akin to Theorem~\ref{a}. 

The issue in question arises in connection with the method of ``arguing in the generic extension'' $M^{\mathbb P, G}$ to prove that a sentence of the forcing language is forced by a condition $p$. Actually, one does not argue in $M^{\mathbb P, G}$ itself, but rather in a theory appropriate to it, without assuming that $G$ actually exists, and uses the existence of the argument to infer the forcing relationship. Thus, this method is applicable in particular to $\forces^{\mathbb P}$, i.e., to $\forces^{V, \mathbb P}$, even though $V$-generic filters demonstrably do not exist. 

Rather than working in the forcing language $\lang L^{M, \mathbb P}$ \emph{per se}, we will work in a more conventional language, with a signature $\splain^*$ that extends the signature $\splain$ of set theory by the addition of a unary predicate symbol $\Vcheck$ and constant symbols $\Pcheck$ and $\Gcheck$.
\m{aj}
\begin{display}
\label{aj}
Let $\Theta$ be the $\splain^*$-theory consisting of
\begin{denumerate}
\item $\tzf$ with the additional predicate symbol $\Vcheck$;
\item \bmob$\Vcheck$ is transitive and contains every ordinal\emob, i.e., \bmob$\forall x, y\, \big(\Vcheck(x) \cand y \in x \cimplies \Vcheck(y)\big)$ and $\forall_{\Ord} \alpha\,\,\Vcheck(\alpha)$\emob;
\item $\tzf^{\Vcheck}$, i.e., all axioms of $\tzf$ relativized to $\Vcheck$;
\item \bmob$\Vcheck(\Pcheck)$ and $\Pcheck$ is a partial order\emob;
\item \bmob$\Gcheck$ is a $\Vcheck$-generic filter on $\Pcheck$\emob;
\item \bmob{}every set is $x^{\Gcheck}$ for some $x \in \Vcheck^{\Pcheck}$\emob.
\end{denumerate}
\end{display}
The following proposition is a key element of the theory of generic extensions.
\m{bd}
\begin{proposition}
\label{bd}
{\normalfont [$\tzf$]}
For any finite subset $T$ of $\Theta$ there is a finite subset $F$ of $\tzf$ such that for any transitive (set) model $M$ of $F$, partial order $\mathbb P \in M$, and $M$-generic filter $G$ on $\mathbb P$, $M[G] \models T$ with $\Vcheck, \Pcheck, \Gcheck$ interpreted respectively as $M$, $\mathbb P$, and $G$.
\end{proposition}

At this point, some remarks concerning the definability of the forcing relation are in order. Since $\forces^{M, \mathbb P}$ subsumes the satisfaction relation for a transitive class $M$, the full forcing relation---like the full satisfaction relation for $M$---is not definable over $M$, and $\forces^{\mathbb P}$ is not definable in the context of $\tzf$. Instead, we define, for each formula $\phi$, the relation
\m{bb}
\refstepcounter{theorem}\begin{equation}
\label{bb}
\{ \la p, x_0, \dots, x_{n-1}\ra \mid p \in |\mathbb P| \cand x_0, \dots, x_{n-1} \in V^{\mathbb P} \cand p \forces^{\mathbb P} \bmob\obsub\phi(x_0, \dots, x_{n-1})\emob \}.
\end{equation}
This is, of course, a definition schema, not a single definition.\footnote{Alternatively, we may extend $\tzf$ by the addition of a new predicate symbol \bmob$\forces$\emob, with axioms that correspond to the usual recursive definition of the forcing relation. Note that these axioms allow us to generate a definition for the relation (\ref{bb}) for any given $\phi$, but this definition has quantifier depth that increases with that of $\phi$, and the axioms do not yield a definition of $\forces^{\mathbb P}$ in its entirety. Note also that, since \bmob$\forces$\emob{} is not introduced by definition, we must explicitly extend the axiom schemas of $\tzf$ to formulas that incorporate the new symbol. It is not hard to show that this theory is a conservative extension of $\tzf$, so it is largely immaterial which approach we use to the description of forcing over $V$, but in this article we will deal with $\tzf$ unmodified.} In $\tgb$ we have the option of defining a universal forcing predicate analogous to the universal satisfaction predicate $\gbmodels$, but this is irrelevant to the present discussion, so we defer this definition for now.

The following proposition establishes the method of ``arguing in a generic extension''. The theorem is well known, so we only briefly sketch the proof.
\m{av}
\begin{proposition}
\label{av}
{\normalfont [$\tS$]}
Suppose $\psi$ and $\phi$ are $\splain$-formulas with $n+2$ and $n$ free variables, respectively.\n{Let $\hat \phi$ be the canonical name of $\phi$.}  Suppose $\Theta \proves$
{\normalfont
\m{aw}
\begin{display}
\label{aw}
\bmob{}for all $p \in |\Pcheck|$ and $x_0, \dots, x_{n-1} \in \Vcheck^{\Pcheck}$, if $\obsub{\psi^{\Vcheck}}(\Pcheck, p, x_0, \dots, x_{n-1})$ then $p \in \Gcheck \cimplies \obsub{\phi}(x_0^{\Gcheck}, \dots, x_{n-1}^{\Gcheck})$\emob.
\end{display}
}
Then $\tzf \proves$
{\normalfont
\m{ax}
\begin{display}
\label{ax}
\bmob{}if $\mathbb P$ is a partial order, $p \in |\mathbb P|$, $x_0, \dots, x_{n-1} \in V^{\mathbb P}$, and $\obsub{\psi}(\mathbb P, p, x_0, \dots,\linebreak[1] x_{n-1})$, then $p \forces^{\mathbb P} \bmob\obsub{\phi}(x_0, \dots, x_{n-1})\emob$\emob.
\end{display}
}
\end{proposition}
\begin{proof} Let $\psi$ and $\phi$ be given, and let $\Theta_0$ be a finite subset of $\Theta$ such that $\Theta_0$ proves (\ref{aw}). Let $F$ be a finite subset of $\tzf$ such that for any transitive model $M$ of $F$, and any partial order $\mathbb P \in M$, 
\begin{denumerate}
\item $M$ correctly defines $\forces^{M, \mathbb P} \phi$; and
\item for any $M$-generic filter $G$ on $\mathbb P$, $M[G] \models \Theta_0$ with $\Vcheck, \Pcheck, \Gcheck$ interpreted respectively as $M$, $\mathbb P$, and $G$. 
\end{denumerate}
We now sketch a proof of (\ref{ax}) in $\tzf$ (without being too fussy about use vs. mention). We begin by supposing toward a contradiction that it is not the case. We use a reflection argument, followed by the transitive collapse of a countable elementary substructure, to obtain a countable transitive model $M$ of $F$, with a partial order $\mathbb P \in M$, $p \in |\mathbb P|$, and $x_0, \dots, x_{n-1} \in M^{\mathbb P}$, such that $M \models\psi[\mathbb P, p, x_0, \dots, x_{n-1}]$, and $p \nforces^{M, \mathbb P} \phi(x_0, \dots, x_{n-1})$. We let $G$ be an $M$-generic filter on $\mathbb P$ such that $p \in G$ and $M[G] \nmodels \phi[x^G_0, \dots, x^G_{n-1}]$. Since $M \models F$, $M[G] \models \Theta_0$. It follows that $M[G] \models \phi[x^G_0, \dots, x^G_{n-1}]$, a contradiction. 
\n{ 
Let $\Psi$ be a finite set of $\splain$-formulas including $F$ and containing $\psi$ and the formula \bmob$\forces \obsub\phi$\emob. As usual, let $\hat \Psi$ be the canonical name of $\Psi$. The following argument is a proof in $\tzf$ of (\ref{ax}).

\bmob{}Suppose $\mathbb P'$ is a partial order, $p' \in |\mathbb P'|$, $x'_0, \dots, x'_{n-1} \in V^{\mathbb P'}$, and $\obsub{\psi}(\mathbb P', p', x'_0, \dots, x'_{n-1})$. Suppose toward a contradiction that $p' \nforces^{\mathbb P'} \bmob\obsub{\phi}(x'_0, \dots, x'_{n-1})$\emob. Let $\alpha$ be an ordinal such that $\mathbb P', x'_0, \dots, x'_{n-1} \in V_\alpha$ and $V_\alpha$ is a $\obsub{\hat\Phi}$-elementary substructure of $V$.\footnote{In other words, for each $\nu \in \subclose\Psi$, if $\Free \nu = \{ u_0, \dots, u_{m-1} \}$, we impose the requirement \bmob$\forall \obsub{u_0, \dots, u_{m-1}} \in V_\alpha\, (\obsub\nu^{V_\alpha} \ciff \obsub\nu)$\emob.} Let $M'$ be a countable elementary substructure of $V_\alpha$ such that $\mathbb P', p', x'_0, \dots, x'_{n-1} \in M'$. Let $M$ be the transitive collapse of $M'$, and let $\mathbb P$, $p$, $x_0, \dots, x_{n-1}$ be the images under the collapsing map of the corresponding elements of $M'$. Then $M \models \obsub{\hat F}$, $M \models \obsub\psi[\mathbb P, p, x_0, \dots, x_{n-1}]$, and $M \models$ \bmob$\obass p \nforces^{\obass{\mathbb P}} \bmob\obsub{\phi}(\obass{x_0, \dots, x_{n-1}})\emob$. Let $G$ be an $M$-generic filter on $\mathbb P$ such that $p \in G$ and $M[G] \nmodels \bmob\obsub\phi[x^G_0, \dots, x^G_{n-1}]$. Then $M[G] \models \obass{\hat{\Theta_0}}$. It follows that $M[G] \models \bmob\obsub\phi[x^G_0, \dots, x^G_{n-1}]$, a contradiction.\emob
}

In effect, we justify the method of arguing in a hypothetical generic extension of $V$ by arguing in an actual generic extension of a countable transitive model of a finite fragment of $\tzf$.
\end{proof}

Note that $\Theta$ implements the hypothesis that $\Vcheck$ is a model of $\tzf$ by positing each axiom of $\tzf$ relativized to $\Vcheck$. In a pure set theory we have no other option, as $\tzf$ is not finitely axiomatizable, and proper classes do not exist. In $\tgb$ we may use the satisfaction predicate $\gbmodels$ defined in (\ref{t}) to implement the hypothesis that a proper class $M$ is a model of $\tzf$ as the single sentence \bmob$M \gbmodels \tzf$\emob.

The use of $\gbmodels$ can play the same simplifying role in the exposition of the theory of forcing as in other areas of set theory that deal with proper class models, but in the case of forcing, the following question arises:

Let $\cplain^*$ be the signature $\cplain$ with additional constants $\Vcheck$, $\Pcheck$, and $\Gcheck$; we also treat $\Vcheck$ as a unary predicate in the usual way.
\m{ak}
\begin{display}
\label{ak}
Let $\Theta'$ be the $\cplain^*$-theory which is $\Theta$ with the following changes:
\begin{denumerate}
\item [$1'$.] $\tgb$.
\item [$3'$.] \bmob$\Vcheck \gbmodels \tzf$\emob.
\end{denumerate}
\end{display}
When ``arguing in a generic extension'', we will naturally reason from $\Theta'$ rather than $\Theta$. The question is whether Proposition~\ref{av} applies with $\Theta'$ in place of $\Theta$.

In the absence of an affirmative answer to this question, the usefulness of $\gbmodels$ is much diminished, as one must maintain a parallel development of forcing without $\gbmodels$ to use when deriving forcing relations by ``arguing in a generic extension''. Thus, the following theorem is a great convenience.

\m{be}
\begin{theorem}
\label{be}
{\normalfont [$\tS$]}
Suppose $\psi$ and $\phi$ are $\splain$-formulas with $n+2$ and $n$ free variables, respectively.\n{Let $\hat \phi$ be the canonical name of $\phi$.}  Suppose $\Theta' \proves$
\m{bf}
\begin{display*}
\label{bf}
\bmob{}for all $p \in |\Pcheck|$ and $x_0, \dots, x_{n-1} \in \Vcheck^{\Pcheck}$, if $\obsub{\psi^{\Vcheck}}(\Pcheck, p, x_0, \dots, x_{n-1})$ then $p \in \Gcheck \cimplies \obsub{\phi}(x_0^{\Gcheck}, \dots, x_{n-1}^{\Gcheck})$\emob.
\end{display*}
Then $\tzf \proves$
\m{bg}
\begin{display*}
\label{bg}
\bmob{}if $\mathbb P$ is a partial order, $p \in |\mathbb P|$, $x_0, \dots, x_{n-1} \in V^{\mathbb P}$, and $\obsub{\psi}(\mathbb P, p, x_0, \dots, x_{n-1})$, then $p \forces^{\mathbb P} \bmob\obsub{\phi}(x_0, \dots, x_{n-1})\emob$\emob.
\end{display*}
\end{theorem}
\begin{proof}
This follows immediately from Theorem~\ref{av} and the following theorem (\ref{ai}).
\end{proof}

\m{ai}
\begin{theorem}
\label{ai}
{\normalfont [$\tS$]}
$\Theta'$ is a conservative extension of $\Theta$ in the sense that for any $\splain^*$-sentence $\sigma$, if $\Theta' \proves \sigma$ then $\Theta \proves \sigma$.
\end{theorem}
\begin{proof} We will carry out the proof in $\tC$. Since the statement of the theorem is an $\splain$-sentence and $\tC$ is a conservative extension of $\tS$, the theorem follows from $\tS$.

We begin as in the finitary proof of Theorem~\ref{a}. Thus, suppose $\Theta' \proves \sigma$ and suppose toward a contradiction that $\Theta \nproves \sigma$. Let $\str M = (M; \in^{\str M}, M_0, \mathbb P, G)$ be a satisfactory structure such that $\str M \models \Theta \cup \{ \opneg \sigma\}$, where $M_0 = \Vcheck^{\str M}$, $\mathbb P = \Pcheck^{\str M}$ and $G = \Gcheck^{\str M}$. Let $T$ be the (full) satisfaction relation for $\str M$. Thus, $\models^T \Theta \cup \{ \opneg \sigma\}$. Let $\splain^{\Vcheck}$ be the expansion of the signature $\splain$ by the addition of the unary predicate symbol $\Vcheck$ (without the constant symbols $\Pcheck$ and $\Gcheck$ of $\splain^*$). Extend $\str M$ to a $\cplain^*$-structure $\str M'$ as before, by adding ``proper classes'' definable over $\str M$.
\m{bi}
\begin{display}
\label{bi}
Clearly, each added class is defined by an $\splain^{\Vcheck}$-formula from a parameter in $M$, which may incorporate $\mathbb P$ and/or $G$.\footnote{For notational simplicity, any formula requiring $n > 1$ parameters is replaced by a formula with a single parameter, which is an $n$-sequence.}
\end{display}
Let $\Psi$ be the class of $\cplain^*$-formulas with class-quantifier depth at most 2, and let $T'$ be the $\Psi$-satisfaction relation for $\str M'$.
\m{al}
\begin{claim}
\label{al}
$\models^{T'} \Theta'$.
\end{claim}
{\sc Proof.} It is straightforward to check that $\models^{T'} \theta$ for all $\theta \in \Theta'$ other than \bmob$\Vcheck \gbmodels \tzf$\emob. To show that $\models^{T'} $ \bmob$\Vcheck \gbmodels \tzf$\emob, suppose toward a contradiction that it does not. Note that \bmob$\Vcheck \gbmodels \tzf$\emob{} is \bmob{}for every $\theta \in \tzf$, for every $\{\theta\}$-satisfaction relation $S$ for $\Vcheck$, $\la \theta, 0 \ra \in S$\emob, so it has class-quantifier depth 1. Hence $\models^{T'} \bmob\Vcheck \ngbmodels \tzf\emob$, so there exist $\theta \in M$ and $S \in M'$ such that $\models^T$ \bmob$\obass\theta \in \tzf$\emob{} and $\models^{T'}$ \bmob$\obass S$ is the $\{\obass\theta\}$-satisfaction relation for $\Vcheck$, and $\nmodels^{\obass S} \obass \theta$\emob.

At this point the proof of Theorem~\ref{a} bifurcated according to whether $\theta$ is in the standard or the nonstandard part of $\str M$. In the latter event we obtained a contradiction from the fact that $S$ would include the full satisfaction relation for $\str M$, which cannot be definable over $\str M$. That depended on Definition~\ref{ad} of $\theta'$ in terms of $\Phi^{\splain}_n$-satisfaction relations: if $n \in M$ is nonstandard then $\Phi^{\splain}_n$ contains every standard $\splain$-formula. For this method (with $(\Vcheck; \in)$ and $\theta$ instead of $(M;\in)$ and $\theta'$) to be applicable in the present case, we would have to reformulate the universal satisfaction predicate so that $\str S \gbmodels \phi[A]$ iff for every $\Phi_\phi$-satisfaction relation $S$ for $\str S$, $\la \phi, A \ra \in S$, where $\Phi_\phi$ is defined so that if $\phi$ is nonstandard then $\Phi_\phi$ contains every standard formula. For example, we could let $\Phi_\phi$ consist of all formulas with complexity not greater than that of $\phi$ in some appropriate sense, rather than letting $\Phi_\phi$ consist of all subformulas of $\phi$, as we have done. In the proof of Theorem~\ref{a} it was legitimate to define $\Phi_\phi$ however we wished, as the theorem does not mention satisfaction. In the present case, such an alteration would be inelegant, to say the least---and it is unnecessary, since we may proceed as follows.

As before, let $H$ be the standard part of $\HF^{\str M}$. Note that $H$ is also the standard part of $\HF^{\str M_0}$, where $\str M_0$ is the substructure of $\str M$ corresponding to $M_0$. Let $x \mapsto \bar x$ be the isomorphism of $H$ with $\HF$. To simplify the notation, suppose that $(\lang L^{\splain^*})^{\str M}$ is $\lang L^{\splain^*}$, so $\bar \epsilon = \epsilon$ for any $\splain^*$-expression $\epsilon$.

It is easy to show (as in Theorem~\ref{c}) that if $\theta \in H$ then $\models^{T'}$ \bmob$\models^{\obass S} \obass\theta$\emob{} iff $\models^T \theta^{\Vcheck}$, so $\models^{T'}$ \bmob$\models^{\obass S} \obass\theta$\emob, since  $\models^T \tzf^{\Vcheck}$ by hypothesis. Thus, $\theta$ is in the nonstandard part of $\tzf^{\str M}$, which means that it is an instance of one of the axiom schemas for a nonstandard formula.

We will suppose that the schemas are \axiom{Collection} and \axiom{Comprehension}. (\axiom{Foundation} may be formulated in $\tzf$ as the statement that all nonempty sets have an $\in$-minimal element; it need not be formulated as a schema.) A sufficiently general version of \axiom{Collection} is
\[
\bmob\forall y\, \forall x\, \exists_{\Ord} \alpha\, \forall z \in x\, (\exists_{\Ord} \beta\,\, \obsub\psi(z, \beta, y) \cimplies \exists \beta < \alpha\,\,  \obsub\psi(z, \beta, y) ),
\]
where $\psi$ is an $\splain$-formula with three free variables. Suppose $\theta$ is the above instance of \axiom{Collection}. We will derive a contradiction by showing that $\models^{T'}$ \bmob$\models^{\obass S} \obass \theta$\emob.

Since $\models^{T'}$ \bmob$\obass S$ is the $\{\obass\theta\}$-satisfaction relation for $\Vcheck$\emob{} and $\models^T$ \bmob$\obass\psi$ is a subformula of $\obass\theta$\emob, it suffices to show that
\begin{multline*}
\models^{T'} \bmob\forall y, x \in \Vcheck\, \exists_{\Ord} \alpha\,\forall z \in x\, \big(\exists_{\Ord} \beta\, \models^{\obass S} \obass\psi[z, \beta, y]\\
 \cimplies \exists_{\Ord} \beta < \alpha\, \models^{\obass S} \obass\psi[z, \beta, y]\big)\emob.
\end{multline*}
This is an instance of the \axiom{Collection} schema of $\tgb$ and follows from the fact that $\models^{T'}$ \axiom{Collection}.

The case that $\theta$ is an instance of \axiom{Comprehension} is not so simple. Suppose
\[
\theta = \bmob\forall y\, \forall x\, \exists x'\,\forall z\, ( z \in x' \ciff z \in x \cand \obsub\psi(z,y))\emob,
\]
where $\psi$ is an $\splain$-formula with two free variables. Given $y, x \in M_0$, we must show that there exists $x' \in M_0$ such that for all $z \in M_0$, $z \in^{\str M} x'$ iff $z \in^{\str M} x$ and $\models^{T'}$ \bmob$\models^{\obass S} \obass \psi [\obass{z,y}]$\emob.

By construction, $S$ represents a subclass of $M$ definable over $\str M$ by an $\splain^{\Vcheck}$-formula $\phi$ from a parameter in $M$, which is $a^G$ for some $a \in M_0$, i.e.,
\[
\models^{T'} \bmob\models^{\obass S} \obass\psi[\obass A]\emob \ciff \la \psi, A \ra \in^{\str M'} S \ciff \models^T \phi[\psi, A, a^G].
\]
Thus, given $y, x \in M_0$,
\m{bp}
\begin{display}
\label{bp}
we must show that there exists $x' \in M_0$ such that for all $z \in M_0$, $z \in^{\str M} x'$ iff $z \in^{\str M} x$ and $\models^{T} \phi[\psi, A, a^G]$, where $A$ is the assignment of $z$ and $y$ to the free variables of $\psi$.
\end{display}

Let $u, v$ be new variables, and let $\phi'$ be the $\splain^{\Vcheck}$-formula with free variables $u, v$, obtained from \bmob$S$ is the $\{\obsub u \}$-satisfaction relation for $\Vcheck$\emob{} by replacing each subformula of the form \bmob$\la \psi, A \ra \in S\emob$ by $\phi(\psi, A, v)$. Without belaboring the issue, suffice it to say that $\phi'$ is a conjunction of formulas such as
\begin{denumerate}
\item \bmob$\obsub u$ is an $\splain^*$-formula\emob;
\item \bmob{}for any subformulas $\psi$ and $\psi'$ of $\obsub u$ and $\Vcheck$-assignment $A$ for $\psi$, if $\psi = \opneg \psi'$, then $\obsub\phi(\psi, A, v)$ iff $\cneg\obsub\phi(\psi', A, v)$\eob{} (with similar formulas for the other propositional connectives); and
\item \bmob{}for any subformulas $\psi$ and $\psi'$ of $\obsub u$, variable $w$, and $\Vcheck$-assignment $A$ for $\psi$, if $\psi = \opexists w\,\psi'$, then $\obsub\phi(\psi, A, v)$ iff for some $x$ such that $\Vcheck(x)$, $\obsub\phi(\psi', A \cup \{ (w, x) \}, v)$\eob{} (with a similar formula for the universal quantifier).
\end{denumerate}
We now have
\[
\models^T \phi'[\theta, a^G].
\]
Since $T$ is the full satisfaction relation for $\str M$ and $\models^T \Theta$, any deduction from $\Theta$ holds in $T$. We will therefore argue in $\Theta$ as follows.

\bmob{}Suppose $\obsub{\phi'(\theta, a^G)}$. Let $p \in G$ be such that
\[
p \forces \bmob\obsub{\phi'(\check\theta, a)}\emob.
\]
\m{an}
\begin{claim}
\label{an}
For every subformula $\psi$ of $\theta$ and every $\Vcheck$-assignment $A$ for $\psi$, $p$ decides \bmob$\obsub{\phi(\check \psi, \check A, a)}$\emob, i.e., either
\begin{denumerate}
\item $p \forces \bmob\obsub{\phi(\check \psi, \check A, a)}\emob$, or
\item $p \forces \bmob\obsub{\opneg \phi(\check \psi, \check A, a)}\emob$.
\end{denumerate}
\end{claim}
{\sc Proof.} Suppose not. Let $\psi$ be a counterexample of minimal complexity. By way of illustration, suppose $\psi = \opexists w\, \psi'$, and suppose $A$ is a $\Vcheck$-assignment for $\psi$. For any $x \in \Vcheck$, let $A^x = A \cup \{ (w,x) \}$, the extension of $A$ that assigns $x$ to $w$. By hypothesis, for any $x \in \Vcheck$, $p$ decides \bmob$\obsub{\phi(\check \psi', \check{A^x}, a)}$\emob. Recall that $p$ forces \bmob$\obsub{\phi'(\check \theta, a)}\emob$, which says that \bmob$\obsub\phi$\emob{} defines the $\{\theta\}$-satisfaction relation for $\Vcheck$ from the parameter $a^G$.

Suppose $p \nforces \bmob\obsub{\phi(\check \psi, \check A, a)}\emob$. Then for all $x \in \Vcheck$, $p \nforces \bmob\obsub{\phi(\check{\psi'}, \check{A^x}, a)}\emob$. Thus, for all $x \in \Vcheck$, since $p$ decides \bmob$\obsub{\phi(\check \psi', \check{A^x}, a)}$\emob, $p \forces \bmob\obsub{\opneg\phi(\check{\psi'}, \check{A^x}, a)}\emob$. Hence, $p \forces \bmob\obsub{\opneg \phi(\check \psi, \check A, a)}\emob$.

The other recursive clauses in the definition of satisfaction are handled similarly, and the atomic formulas are easily dealt with.\qed (\ref{an})

Let $x'$ be the set of $z \in x$ such that $p \forces \bmob\obsub{\phi(\check \psi, \check A, a)}\emob$, where $A$ is the assignment of $z$ and $y$ to the free variables of $\psi$. Then $x' \in \Vcheck$ by virtue of $\axiom{Comprehension}^{\Vcheck}$. Given $z \in \Vcheck$, let $A$ be the assignment of $z$ and $y$ to the free variables of $\psi$. If $z \in x'$ then $p \forces \bmob\obsub{\phi(\check \psi, \check A, a)}$\emob, so $\obsub\phi(\psi, A, a^G)$, since $p \in G$. On the other hand, if $z \notin x'$ then $p \nforces \bmob\obsub{\phi(\check \psi, \check A, a)}$\emob, so by the claim, $p \forces \bmob\obsub{\opneg\phi(\check \psi, \check A, a)}$\emob, whence $\cneg\obsub{\phi}(\psi, A, a^G)$, since $p \in G$.\emob

As noted above, the existence of this argument in $\Theta$ shows that there exists $x' \in M_0$ as required by (\ref{bp}), and this completes the proof that $\models^{T'} \theta$.\qed (\ref{al})
 
We now know that $\models^{T'} \Theta'$ and $\models^{T'} \opneg \sigma$, where $T'$ is the $\Psi$-satisfaction relation for $\str M'$, $\Psi$ being the class of $\cplain^*$-formulas with class-quantifier depth at most 2. We now wish to derive a contradiction from the assumption that $\Theta' \proves \sigma$. As in the proof of Theorem~\ref{a}, if we had \axiom{Infinity} we could arrange that $\str M'$ be a set and take $T'$ to be the full satisfaction relation for $\str M'$, from which the desired contradiction would follow at once. It would not be inappropriate to finish this way, as the theorem is only of interest in the context of forcing, which is only of interest in the context of \axiom{Infinity}; however, an argument can be made that if a finitary theorem has a finitary proof, one should be given, and we oblige.

To complete the proof in $\tC$ we proceed as in finitary proof of Theorem~\ref{a}, showing that any proof $\pi$ of $\sigma$ from $\Theta'$ may be replaced by a proof $\pi'$ of $\sigma$ from $\Theta$. Let $\theta' = \bmob\Vcheck \gbmodels \tzf\emob$, i.e.,
\m{bj}
\begin{display}
\label{bj}
\bmob{}for all $S$, for all $x \in \tzf$, if $S$ is an $\{x\}$-satisfaction relation for $\Vcheck$, then $\la x, 0 \ra \in S$\emob.
\end{display}
As before, we eliminate class variables in favor of class constants, and then replace each expression $\tau \opin C$ by $\phi_C(\tau)$, where $\phi$ is an appropriate ``definition'' of $C$. Each axiom of $\tgb$ used as a premise in $\pi$ is replaced in $\pi'$ by finitely many instances of axioms of $\tzf$.

The premise $\theta'$ is replaced by finitely many sentences $\theta^\phi$ obtained from $\theta'$ by omitting the universal quantification of $S$ and replacing each expression $\tau \opin S$ in (\ref{bj}) by $\phi(\tau)$, where $\phi$ is a formula with one free variable $v$. As before, since each such $\phi$ defines an element of $M'$ and $\models^{T'} \theta'$, it follows that $\models^{T'} \theta^\phi$. Hence $\str M \models \theta^\phi$.

Thus, we have an $\splain^*$-proof of $\sigma$ from premises that are true in $\str M$, so $\str M \models \sigma$, contradicting our assumption that $\str M \models \opneg \sigma$.
\end{proof}

\section{The universal forcing relation}
\m{bo}\label{bo}
We conclude by giving the promised definition of the universal forcing relation and valuation function. We leave it to the reader to supply the definitions of `$\Phi^{M, \mathbb P}$-forcing relation' and `$\Phi^{M, \alg A}$-valuation function', where $\Phi$ is a class of $\splain^{\Vcheck}$-formulas. (These will cover all sentences obtained from subformulas of members of $\Phi$ by substitution of elements of $M^{\mathbb P}$ or $M^{\alg A}$, respectively, for their free variables.) 
\m{ag}
\begin{definition}
\label{ag}
{\normalfont [$\tgb$]}
Suppose $M$ is a transitive model of $\tzf$, $\mathbb P$ is a partial order in $M$, and $\alg A$ is an $M$-complete boolean algebra in $M$.
\begin{denumerate}
\item Suppose $\phi$ is an $\lang L^{M, \mathbb P}$-sentence and $p \in |\mathbb P|$. Then $p \forces^{*M,\mathbb P} \phi$ \tiffdef for every $\{\phi\}^{M, \mathbb P}$-forcing relation $F$, $p \binop F \phi$.
\item Suppose $\phi$ is an $\alg L^{M, \alg A}$-sentence. If there exists a $\{\phi\}^{M, \alg A}$-valuation function $F$ then $\bv \phi^{*M, \alg A} \eqdef F\f \phi$; otherwise, $\bv \phi^{*M, \alg A} \eqdef \boldsymbol1$.\footnote{$\boldsymbol1$ is the correct value for $\bv \phi^{*M, \alg A}$ if no $\{\phi\}^{M, \alg A}$-valuation function exists, because in this case no $\{\phi\}^{M, \mathbb P}$-forcing relation exists, so $\{ p \mid p \forces^{*M,\mathbb P} \phi \} = |\mathbb P|$, which corresponds to boolean value $\boldsymbol1$.} 
\end{denumerate}
\end{definition}
These definitions reduce to the usual ones when $M$ is a set, and when $M$ is a proper class they permit the development of the theory of forcing in the usual way.

As we have noted above, this universal forcing predicate is not involved in any of the considerations of the preceding section. If we wished, of course, we could reformulate the conclusion of Theorem~\ref{be} to be that $\tgb \proves$
\m{bk}
\begin{display*}
\label{bk}
\bmob{}if $\mathbb P$ is a partial order, $p \in |\mathbb P|$, $x_0, \dots, x_{n-1} \in V^{\mathbb P}$, and $\obsub{\psi}(\mathbb P, p, x_0, \dots, x_{n-1})$, then $p \forces^{*\mathbb P} \bmob\obsub{\phi}(x_0, \dots, x_{n-1})\emob$\emob,
\end{display*}
which would serve the same practical purpose, and would be appropriate in an exposition of the theory of forcing in $\tgb$ using $\forces^*$.



\begin{thebibliography}{10}

\bibitem{Bernays:1937}
Paul Bernays.
\newblock A system of axiomatic set theory---{P}art {I}.
\newblock {\em Journal of Symbolic Logic}, 2(1):65--77, 1937.

\bibitem{Craig:1958}
W.~Craig and R.~L. Vaught.
\newblock Finite axiomatizability using additional predicates.
\newblock {\em Journal of Symbolic Logic}, 23(3):289--308, 1958.

\bibitem{Jech:2003}
Thomas~J. Jech.
\newblock {\em Set {T}heory}.
\newblock Springer Monographs in Mathematics. Springer-Verlag, New York, third
  edition, 2003.

\bibitem{Kleene:1952}
Stephen~C. Kleene.
\newblock Finite axiomatizability of theories in the predicate calculus using
  additional predicate symbols.
\newblock {\em Memoirs of the American Mathematical Society}, (10):27--68,
  1952.

\bibitem{Levy:1979}
Azriel Levy.
\newblock {\em Basic {S}et {T}heory}.
\newblock Springer-Verlag, New York, first edition, 1979.

\bibitem{Ryll:1952}
Czes{\l}aw Ryll-Nardzewski.
\newblock The role of the axiom of induction in elementary arithmetic.
\newblock {\em Fundamenta Mathematicae}, 39:239--263, 1952.

\bibitem{Shoenfield:1954}
Joseph~R. Shoenfield.
\newblock A relative consistency proof.
\newblock {\em Journal of Symbolic Logic}, 19(1):21--28, 1954.

\bibitem{Solovay:1978}
Robert~M. Solovay, W.~N. Reinhardt, and A.~Kanamori.
\newblock Strong axioms of infinity and elementary embeddings.
\newblock {\em Annals of Mathematical Logic}, 13(1):73--116, 1978.

\bibitem{Takeuti:1975}
Gaisi Takeuti.
\newblock {\em Proof {T}heory}, volume~81 of {\em Studies in Logic and the
  Foundations of Mathematics}.
\newblock Elsevier, New York, 1975.

\bibitem{RVWwtbi1:2012}
Robert~A. Van~Wesep.
\newblock \href{http://www.mathetal.net/books.php}{Foundations of {M}athematics: {A} {G}eneralist's {G}uide}.

\end{thebibliography}

\end{document}